\documentclass[11pt,
]{amsart}

\usepackage{amssymb}
\usepackage{latexsym}
\usepackage{amsmath}
\usepackage{amsthm}
\usepackage{amsfonts}
\usepackage{color}

\usepackage[draft=false,setpagesize=false,pdftex,pdfstartview=FitH,
colorlinks=true,citecolor=blue,pagebackref=true,bookmarks=false]{hyperref}

\numberwithin{equation}{section}

\theoremstyle{definition}
\newtheorem{defi}{Definition}[section]
\theoremstyle{plain}
\newtheorem{prop}[defi]{Proposition}
\newtheorem{teo}[defi]{Theorem}

\newtheorem{lema}[defi]{Lemma}

\theoremstyle{remark}

\newcommand{\R}{\mathbf{R}}
\newcommand{\RN}{\mathbf{R}^N}
\newcommand{\N}{\mathbf{N}}

\newcommand{\rd}{\mathrm{d}}
\newcommand{\cM}{\mathcal{M}_{\lambda,\Lambda}}

\newcommand{\cS}{\mathcal{S}}

\begin{document}

\title[]{Existence and nonexistence of positive solutions to some fully nonlinear equation 
in one dimension}

\author[]{Patricio Felmer}
\address{
Patricio Felmer - 
Departamento de Ingenier\'\i a Matem\'atica and CMM (UMI 2807 CNRS), Universidad de Chile, Casilla 170 Correo 3, 
Santiago, Chile
\newline {\tt pfelmer@dim.uchile.cl}.
}

\author[]{Norihisa Ikoma}
\address{
Norihisa Ikoma - 
Faculty of Mathematics and Physics, 
Institute of Science and Engineering, 
Kanazawa University, 
Kakuma, Kanazawa, Ishikawa 9201192, JAPAN
\newline {\tt ikoma@se.kanazawa-u.ac.jp}.
}

\date{\today}

\begin{abstract} 
In this paper, we consider the existence (and nonexistence) of solutions to 
	\[
		-\cM^\pm (u'') + V(x) u = f(u) \quad {\rm in} \ \R
	\]
where $\cM^+$ and $\cM^-$ denote the Pucci operators with $0< \lambda \leq \Lambda < \infty$, 
$V(x)$ is a bounded function, 
$f(s)$ is a continuous function and its typical example is 
a power-type nonlinearity $f(s) =|s|^{p-1}s$ $(p>1)$. 
In particular, we are interested in positive solutions which decay at infinity, 
and the existence (and nonexistence) of such solutions is proved. 
\end{abstract}
\subjclass[2010]{}
\maketitle

\section{Introduction}\label{section:1}

In this paper, we study the existence and nonexistence of solutions to the following nonlinear differential equations 
%for the extremal Pucci operators
	\begin{equation}\label{eq:1.1}
		-\cM^\pm (u'') + V(x) u = f(u) \quad {\rm in} \ \R, \quad 
		u > 0 \quad {\rm in} \ \R, \quad 
		\lim_{|x| \to \infty }u (x) = 0.
	\end{equation}
%and 
%	\begin{equation}\label{eq:1.2}
%	-\cM^- (u'') + V(x) u = f(u) \quad {\rm in} \ \R, \quad 
%	u > 0 \quad {\rm in} \ \R, \quad 
%	\lim_{|x| \to \infty }u (x) = 0.
%	\end{equation}
Here $V$ and $f$ are given functions, $0<\lambda \leq \Lambda < \infty$ constants and 
$\cM^\pm(s)$ the Pucci operators defined by 
	\[
		\cM^+(s) := 
		\left\{\begin{aligned}
			& \Lambda s & &{\rm if} \ s \geq 0,\\
			& \lambda s & &{\rm if} \ s < 0,
		\end{aligned}\right.
		\quad 
		\cM^-(s) := 
		\left\{\begin{aligned}
			& \lambda s & &{\rm if} \ s \geq 0,\\
			& \Lambda s & &{\rm if} \ s < 0.
		\end{aligned}\right.
	\]
We remark that when $\lambda = \Lambda$, one has $\cM^\pm(u'') = \lambda u''$.

	One of motivations to study equations like \eqref{eq:1.1} %and \eqref{eq:1.2} 
is to see to what extent the properties and the results in the semilinear case 
can be generalized to the fully nonlinear case. 
When $\lambda = \Lambda$, 
\eqref{eq:1.1} %(or \eqref{eq:1.2}) 
is well studied and it is proved that 
\eqref{eq:1.1} has a solution for various $V(x)$ and $f(s)$ 
by critical point theory. Here we refer to \cite{St-08,W-96} and references therein.

	On the other hand, when $\lambda \neq \Lambda$, 
\eqref{eq:1.1} %and \eqref{eq:1.2} are
is not studied well. 
In \cite{FQ-04}, instead of \eqref{eq:1.1}, %and \eqref{eq:1.2}, 
the authors study 
the existence of positive radial solutions of 
	\begin{equation}\label{eq:1.3}
		-\cM^\pm ( D^2 u) + \gamma u = f(u) \quad {\rm in} \ B_R(0) \subset \RN, \quad 
		u = 0 \quad {\rm on} \ \partial B_R(0)
	\end{equation}
as well as 
	\[
		-\cM^\pm ( D^2 u) + u = u^p \quad {\rm in} \ \RN.
	\]
Here $N \geq 3$, $0 \leq \gamma$ and $1 < p < p_\ast^\pm$ where 
$p_\ast^\pm$ are critical exponents for $\cM^\pm$ 
(see also \cite{AS-11,CL-00,FQ-02,FQ-03}). 
Recently, in \cite{GLP-17}, the authors show the existence of 
infinitely many radial solutions of \eqref{eq:1.3} when $\gamma = 0$ and 
$f(s) = |s|^{p-1} s$. Moreover, in \cite{GLP-17}, the inhomogeneous case is also
considered and the existence of infinitely many solutions is shown 
on a bounded annulus.

	In this paper, we aim to treat the inhomogeneous equation on the unbounded domain $\R$. 
We emphasis that in general the existence of solutions to \eqref{eq:1.1} %or \eqref{eq:1.2} 
is 
delicate when the equation is inhomogeneous and the domain is unbounded. 
Indeed, we shall prove the nonexistence result when $V(x)$ is monotone. 
See Theorem \ref{defi:1.3} below.

	We first deal with the existence result. 
For $V(x)$, we assume 
	\begin{enumerate}
		\item[(V1)] $V \in W^{1,\infty}(\R)$ and $0 < \inf_{\R} V =: V_0$. 
		\item[{(V2)}] 
		{For a.a. $x \in (-\infty,0)$ and a.a. $y \in (0,\infty)$, 
			$V'(x) \leq 0 \leq V'(y)$}.
		\item[{(V3)}] 
		$V(0) \leq  V_\infty := \lim_{|x|\to \infty} V(x)$ and there exist 
			$C_0,\xi_0 > 0$ such that 
				\[
					\begin{aligned}
						&\text{(for $\cM^+$)} 
						& & (0 \leq) V_\infty  - V(x) 
						\leq C_0 \exp \left( -2 \sqrt{\frac{V_\infty}{\Lambda} + \xi_0 } |x| 
						\right)
						& &{\rm for\ all}\ x \in \R,
						\\
						&\text{(for $\cM^-$)} 
						& & (0 \leq) V_\infty  - V(x) 
						\leq C_0 \exp \left( -2 \sqrt{\frac{V_\infty}{\lambda} + \xi_0 } |x| 
						\right)
						& &{\rm for\ all}\ x \in \R.
					\end{aligned}
				\]
	\end{enumerate}

	Next, for $f(s)$, we suppose the following conditions 
and an example of $f(s)$ is $f(s) = \sum_{i=1}^k a_i s^{p_i}$ 
where $0 < a_i$ and $1 < p_i$: 
	\begin{enumerate}
		\item[(f1)] $f \in C^1(\R)$ and $f(s) = 0 $ for all $s \leq 0$. 
		\item[(f2)] There exists an $\eta_0 > 0$ such that 
		$ \lim_{s \to 0} s^{-1 - \eta_0} f(s) = 0$. 
		\item[(f3)] As $s \to \infty$, 
				\[
					 \frac{f(s)}{s} \to \infty \quad 
					 {\rm and} \quad 
					\frac{f(\theta s)}{f(s)} \to \bar{f} (\theta) \quad 
					{\rm in} \ C_{\rm loc} ((0,1]). 
				\]
		\item[{(f4)}] 
			$ s \mapsto s^{-1} f(s): (0,\infty) \to \R$ is strictly increasing. 
	\end{enumerate}

\medskip

\noindent
{\bf Remark 1.1}	%\begin{remark}
{\it (i) In (f3), it follows that 
$\bar{f} \in C((0,1])$, $\bar{f}(1) = 1$ and $\bar{f}(\theta) \geq 0$ for $\theta \in (0,1]$. 
For example, when $f(s) = s^p$ and $f(s) = s \log s$, one sees 
$\bar{f}(\theta) = \theta^p$ and $\bar{f} (\theta) = \theta$ respectively.

(ii) When $\lambda = \Lambda$, 
condition (f4) is used to obtain bounded Palais-Smale sequences. 
The classical  condition to obtain bounded Palais-Smale sequences 
is the Ambrosetti-Rabinowitz condition: 
$0 < \mu \int_0^s f(t) \rd t \leq f(s) s$ for some $\mu > 2$ and all $s >0$. 
We remark that (f1)--(f4) do not imply this condition. 
In fact, consider a function defined by 
	\[
		f(s) = \eta(s) s^p + (1-\eta(s)) C s \log s
	\]
where $1<p$, $\eta \in C^\infty([0,\infty),\R)$, $\eta'(s) \leq 0$ for every $s \in [0,\infty)$, 
$\eta(s) = 1$ if $0 \leq s \leq 2$, $\eta(s) = 0$ if $3 \leq s$ and 
$C>0$ is chosen so that $ C \log s \geq s^{p-1} $ in $[2,3]$. 
It is easily seen that $f$ satisfies (f1)--(f4) with $\bar{f}(\theta) = \theta$ 
and that
$F(s)$ has the growth $s^2 \log(s)$ as $s \to \infty$, 
providing the required counterexample.

}	%\end{remark}

\medskip

Under these conditions, we have 

\begin{teo}\label{defi:1.2}
	Under \emph{(V1)--(V3)} and \emph{(f1)--(f4)}, 
	\eqref{eq:1.1} %and \eqref{eq:1.2} 
	have a solution. 
\end{teo}

	Next, we turn to the nonexistence result. 
In this case, we assume that $V(x)$ is monotone:
	\begin{enumerate}
		\item[{(V2')}] 
		{$V'(x) \geq 0$ in $\R$ and 
				\[
					\underline{V} = \lim_{x \to - \infty} V(x) 
					< \lim_{x \to \infty} V(x) = \overline{V}. 
				\]}
	\end{enumerate}

Then we have 

\begin{teo}\label{defi:1.3}
	Let $0 < \lambda \leq \Lambda< \infty$ and 
	assume \emph{(V1)}, \emph{(V2')}, \emph{(f1)}, \emph{(f4)} and 
		\begin{equation}\label{eq:1.4}
			\lim_{s\to 0} \frac{f(s)}{s} = 0. 
		\end{equation}
	Then \eqref{eq:1.1} %and \eqref{eq:1.2} 
	have no solution. 
\end{teo}

\medskip

\noindent
{\bf Remark 1.2} %\begin{remark}
{\it	Theorem \ref{defi:1.3} still holds when we replace (V2') by 
		\[
			V'(x) \leq 0 \quad {\rm in}\ \R, \quad 
			\overline{V} = \lim_{x \to - \infty} > \lim_{x \to \infty} V(x) 
			= \underline{V}.
		\]
	}
%\end{remark}

\medskip

	Here we make some comments on the proofs of Theorems \ref{defi:1.2} and \ref{defi:1.3}. 
First, even though equation \eqref{eq:1.1} can be transformed into an equation 
with variational structure (pointed by Professor Evans), 
we prefer to use degree theoretic arguments in view of future applications.  
For Theorem \ref{defi:1.2}, we borrow the idea in \cite{dFLN-82} (cf. \cite{FQ-04}). 
More precisely, we will find a suitable function space $X$ which is a Banach space, 
and rewrite \eqref{eq:1.1} %and \eqref{eq:1.2} 
into the equations 
$( {\rm id} - \mathcal{L}^\pm ) (u) = 0$ 
where ${\mathcal L}^\pm (u) := (-\cM^\pm + V(x))^{-1} f( u(x) )$ for $u \in X$. 
To find a solution $u \neq 0$, we use the Leray--Schauder degree ${\rm deg}_X$ in $X$ 
and prove that 
\begin{enumerate}
	\item[i)] 
		There exists an $r_0>0$ such that 
		${\rm deg}_{X}({\rm id}-{\mathcal L}^\pm, B_{r_0}(0), 0)=1$.
	\item[ii)] 
		There exists an $r_1>r_0$ such that 
		${\rm deg}_{X}({\rm id} - {\mathcal L}^\pm, B_{r_1}(0), 0)=0$.
\end{enumerate}
From i) and ii), we have ${\rm deg}_X( {\rm id} - \mathcal{L}^\pm , A_{r_0,r_1}, 0 ) \neq 0$ 
and find a $u_0 \in A_{r_1,r_2}$ so that $( {\rm id} - \mathcal{L}^\pm ) (u_0) = 0$ 
where $A_{r_1,r_2} := \{ u \in X \ |\  r_1 < \|u\|_X < r_2 \}$. 
One of difficulties here is to find a suitable $X$ in order that 
we can prove the property ii) as well as 
the map $\mathcal{L}^\pm : X \to X$ is compact. 
A key for proving %the property
 ii) is a priori estimates of solutions in $X$. 
Since we treat the unbounded domain, 
we need the uniform decay estimates of solutions as well as the uniform $L^\infty$-bounds. 
This point is different from the bounded domain case and requires %the 
delicate arguments. 
For instance, see Proposition \ref{defi:2.9} below.

	We also point out that the argument of Proposition \ref{defi:2.9} is useful 
to show the nonexistence result namely, Theorem \ref{defi:1.3}. 
Indeed, this case is simpler than Proposition \ref{defi:2.9} and 
we will prove Theorem \ref{defi:1.3} in section \ref{section:3}.

	In Appendix \ref{section:A}, 
we consider \eqref{eq:1.1} % and \eqref{eq:1.2}
 in the special case when $V(x) \equiv {\rm const.} > 0$. 
In this case, we can prove the unique existence of solutions 
up to translations. See Proposition \ref{defi:2.1} and Appendix \ref{section:A}.

%	After this work was completed, 
%Professor Lawrence Craig Evans (personal communication) 
%told us that \eqref{eq:1.1} %and \eqref{eq:1.2} 
%have the variational 
%structure. In the future paper, we will discuss the existence of positive solution 
%via critical point theory. 

\section{Proof of Theorem \ref{defi:1.2}}
\label{section:2}

Throughout this section, we always assume (f1)--(f4) and 
(V1)--(V3). 
%Furthermore, in what follows, instead of writing 
%	\[
%		-\cM^+ (u'') + V(x) u = f(u) \quad {\rm in} \ \R, 
%		\quad 
%		-\cM^- (u'') + V(x) u = f(u) \quad {\rm in} \ \R, 
%	\]
%we use the following abbreviation 
%	\[
%		-\cM^\pm (u'') + V(x) u = f(u) \quad {\rm in} \ \R
%	\]
%when it might be clear from the contexts which equations we consider. 
We begin with the existence result when $V(x) \equiv {\rm const.} > 0$.

\begin{prop}\label{defi:2.1}
Under \emph{(f1)-(f4)}, the equations
	\begin{equation}\label{eq:2.1}
		 \left\{\begin{aligned}
		& -\cM^+ (u'') + V_\infty u = f(u) \quad {\rm in} \ \R, 
		\quad u > 0 \ {\rm in} \ \R, 
		\\
		& u(x) \to 0 \ {\rm as} \ |x| \to \infty, \quad 
		u(0) = \max_{x \in \R} u(x)
		\end{aligned}\right.
	\end{equation}
and 
	\begin{equation}\label{eq:2.2}
		 \left\{\begin{aligned}
			& -\cM^- (u'') + V_\infty u = f(u) \quad {\rm in} \ \R, 
			\quad u > 0 \ {\rm in} \ \R, 
			\\
			& u(x) \to 0 \ {\rm as} \ |x| \to \infty, \quad 
			u(0) = \max_{x \in \R} u(x)
			\end{aligned}\right.
		\end{equation}
	have unique solutions $\omega_+$ and $\omega_-$. 
	Furthermore, there exist $z^\pm > 0$,  $c_1>0$ and $c_2>0$ such that 
		\[
			\begin{aligned}
				&\omega_\pm'' (x) < 0 = \omega_\pm''(z^\pm) < \omega_\pm'' (y) 
				& &\text{\emph{for every $x,y \in \R$ with $|x| < z^\pm < |y|$}},\\
				&
				\omega_\pm(x) \le c_1\exp(-c_2 |x|)
				& &{\rm for \ all} \ x \in \R.
			\end{aligned}
		\]
	Finally, if $u$ satisfies 
		\[
			\begin{aligned}
				&\cM^\pm(u'') + V_\infty u = f(u) \quad {\rm in} \ \R, 
				\quad \ u > 0 \ {\rm in}\ \R,
				\\
				& u(0) = \max_{\R} u, \quad  
				u(x) \to 0 \ {\rm if}\ x \to \infty \ {\rm or} \ x \to -\infty,
			\end{aligned}
		\]
	then $u = \omega_\pm$. 
\end{prop}

We shall prove Proposition \ref{defi:2.1} in Appendix \ref{section:A}.

From now on, we may assume $V(0) < V_\infty$ in (V3) and 
$V(x)$ is not a constant function without loss of generality. 
Under this additional assumption, we fix an $\eta_1>0$ so that
\begin{equation}\label{eq:2.3}
\eta_1<c_2, \quad \Lambda\eta_1^2\left(1+\frac{\eta_0}{2}\right)^2<\frac{V_0}{2},
\end{equation}
where $V_0:=\inf_{\R}V>0$, 
and $\eta_0>0$ and $c_2 >0$ appear in (f2) and Proposition \ref{defi:2.1}. 
We set 
$$
X_{\eta_1}:= \left\{v\in C(\R)\,\Big|\,  \|v\|_{\eta_1}=\sup_{x\in\R}e^{\eta_1|x|}|v(x)|<\infty 
\right\}. 
$$
It is easy to check that $(X_{\eta_1}, \|\cdot \|_{\eta_1})$ is a Banach space.
\begin{lema}\label{defi:2.2}
For every $v\in X_{\eta_1}$, the equations
$$
-\cM^\pm (u'') + V(x) u = f(v(x))\quad {\rm in}\  \R,\quad u\in X_{\eta_1},
$$
have unique solutions. 
\end{lema}
\begin{proof} 
We prove the claim at the same time for $\cM^+$ and $\cM^-$. 
Let $v\in X_{\eta_1}$. For each $n\in\N$, consider
	\[
		\left\{\begin{aligned}
			-\cM^\pm (u'') + V(x) u &= f(v(x))\quad {\rm in} \  (-n,n),\\
			u(-n)=0&=u(n).
		\end{aligned}\right.
	\]
Then the above equations have a unique solution $u_n\in C^2([-n,n])$ 
due to (V1). In fact, since $f(s) \geq 0$ by (f1), (f2) and (f4), 
$u\equiv 0$ is a subsolution of the above equation. 
In addition, one can check that the principal eigenvalues 
of $-\cM^\pm + V(x)$ on $[-n,n]$ with the Dirichlet zero boundary condition 
are positive due to (V1). Thus, by (f2), the positive eigenfunctions 
multiplied by small positive constants become supersolutions and 
so, the  solution $u_n$ is unique.

	Now, the maximum principle yields $u_n \geq 0$ in $[-n,n]$. 
Moreover, we have
$$
\|u_n\|_{L^\infty(-n,n)}\le V_0^{-1}\|f(v)\|_{L^\infty(\R)}.
$$
Indeed, let $x_n\in (-n,n)$ be a maximum point of $u_n$. 
It follows from the equation and $u_n''(x_n) \leq 0$ that
\begin{equation}\label{eq:2.4}
	\| u_n \|_{L^\infty(-n,n)} = u_n(x_n) 
	\le \frac{\|f(v)\|_{L^\infty(\R)}}{V(x_n)}\le \frac{\|f(v)\|_{L^\infty(\R)}}{V_0}.
\end{equation}

	Next we shall show that there exist $C_3>0$ and $\delta_0>\eta_1$ such that
	\begin{equation}\label{eq:2.5}
		u_n(x)\le C_3e^{-\delta_0|x|} \quad 
		\text{for all $x\in \R$ and $n\ge 1$}.
	\end{equation}
To this end, we first notice that (f2) yields
$$
f(s)\le C_4|s|^{1+\eta_0}\quad \text{for all } |s|\le \|v\|_{L^\infty}.
$$
Hence, by the definition of $\|\cdot \|_{\eta_1}$, we obtain
\begin{equation}\label{eq:2.6}
f(v(x))\le C_4 |v(x)|^{1+\eta_0}\le C_4\|v\|_{\eta_1}^{1+\eta_0}e^{-(1+\eta_0)\eta_1|x|} 
\leq C_5 e^{-(1+\eta_0) \eta_1 |x| }
\end{equation}
for all $x \in \R$. 
Recalling \eqref{eq:2.3}, fix an $R_0>0$ so that
\begin{equation}\label{eq:2.7}
-\Lambda\eta_1^2\left(1+\frac{\eta_0}{2}\right)^2 
+ V_0-C_5e^{-\eta_0\eta_1 R_0/2}\ge \frac{V_0}{4}>0.
\end{equation}

	We only treat $n$ with $n>R_0$ and set 
$$
M := 1 + \frac{\|f(v)\|_{L^\infty}}{V_0}, \quad 
\omega_0(x):=Me^{-\left(1+\frac{\eta_0}{2}\right)\eta_1(|x|-R_0)}.
$$
Noting \eqref{eq:2.3}, \eqref{eq:2.6}, \eqref{eq:2.7} and 
\[
	\omega_0''=\left(1+\frac{\eta_0}{2}\right)^2\eta_1^2\omega_0\ge 0,\quad  M\ge 1,
\]
we get the following: for all $R_0\le |x|\le n$,
\begin{eqnarray*}
& &-\cM^\pm(\omega_0'')+V\omega_0-f(v)\\
&\ge&-\cM^\pm (\omega_0'')+V\omega_0-Mf(v)\\
&\ge&\left\{ -\Lambda \left(1+\frac{\eta_0}{2}\right)^2\eta_1^2+V\right\}\omega_0
-MC_5e^{-(1+\eta_0)\eta_1|x|}\\
&\ge&\left[  \left\{ -\Lambda \left(1+\frac{\eta_0}{2}\right)^2\eta_1^2+V\right\}
e^{\left(1+\frac{\eta_0}{2}\right)\eta_1R_0} 
- C_5e^{-\eta_0\eta_1|x|/2} \right]Me^{-(1+\frac{\eta_0}{2})\eta_1|x|}\\
&\ge&\left[   -\Lambda \left(1+\frac{\eta_0}{2}\right)^2\eta_1^2 
+ V_0 - C_5e^{-\eta_0\eta_1R_0/2} \right] 
Me^{-\left(1+\frac{\eta_0}{2}\right)\eta_1|x|}\\
&\ge& 0.
\end{eqnarray*}
Since 
	\begin{equation}\label{eq:2.8}
		\begin{aligned}
			&\cM^+(m_1) - \cM^+(m_2) \geq \cM^-(m_1-m_2), \\
			&\cM^-(m_1) - \cM^-(m_2) \geq \cM^-(m_1-m_2)
		\end{aligned}
	\end{equation}
for all $m_1,m_2 \in \R$, we have 
$$
-\cM^- ( \omega_0'' - u_n'' ) + V(x) (\omega_0 - u_n) \geq 0
$$
for each $R_0\le |x|\le n$. 
From \eqref{eq:2.4} and the definitions of $M$ and $\omega_0$, we have 
$$
u_n(\pm R_0)\le M = \omega_0(\pm R_0), \quad 0=u_n(\pm n)<\omega_0(\pm n).
$$
By the comparison principle, we get
$$
u_n(x)\le \omega_0(x)\quad \mbox{for all} \  R_0\le |x|\le n.
$$
Thus, \eqref{eq:2.5} holds with $\delta_0 := (1+ \eta_0/2) \eta_1$.

	By the elliptic regularity, one sees that $(u_n)$ is bounded in $C^2_{\rm loc}(\R)$, 
hence there exists $(u_{n_k})$ such that $u_{n_k}\to u_0$ in $C^2_{\rm loc}(\R)$, 
where $u_0$ satisfies
$$
-\cM^\pm (u_0'') + V(x) u_0 = f(v(x))\quad\mbox{in }\,\R.
$$
Moreover, from \eqref{eq:2.5}, we obtain
$$
u_0(x)\le C_6e^{-\delta_0|x|}\quad\mbox{in }\,\R.
$$
Since $\delta_0>\eta_1$, $u_0\in X_{\eta_1}$ 
and the existence of solutions is proved.

	For the uniqueness, let $u_1,u_2\in X_{\eta_1}$ be solutions of
$$
-\cM^\pm (u'') + V(x) u = f(v(x))\quad\mbox{in }\,\R
$$
and set $w(x) := u_1(x) - u_2(x)$. 
Then it follows from \eqref{eq:2.8} that $\pm w(x)$ satisfy 
	\[
		- \cM^- (u'') + V(x) u \geq 0 \quad {\rm in} \ \R. 
	\]
Noting that $w(x) \to 0$ as $|x| \to \infty$, 
combining with the above inequality, $\pm w(x)$ do not have 
any negative minimum on $\R$. Hence, $w \equiv 0$ and 
$u_1 \equiv u_2$. 
Thus we complete the proof. 
\end{proof}

\begin{defi}
For $v \in X_{\eta_1}$, we denote by ${\mathcal L}^\pm(v)$ the unique solutions of 
$$
-\cM^\pm (u'') + V(x) u = f(v)  \quad\mbox{in }\,\R, \quad u \in X_{\eta_1}.
$$
\end{defi}

Thanks to Lemma \ref{defi:2.2}, 
${\mathcal L}^\pm: X_{\eta_1}\to X_{\eta_1}$. 
Furthermore, 

\begin{lema}\label{defi:2.4}
The maps ${\mathcal L}^\pm: X_{\eta_1}\to X_{\eta_1}$ are compact.
\end{lema}

\begin{proof}
Let $(v_n)\subset X_{\eta_1}$ be a bounded sequence and 
put $u_n={\mathcal L}^\pm (v_n)$. 
We first show that $(u_n)$ has a convergent subsequence in $X_{\eta_1}.$ Set
$$
M_1=\sup_{n\ge 1}\|v_n\|_{\eta_1}.
$$
Then we have
$$
v_n(x)\le M_1e^{-\eta_1|x|}\quad \mbox{for all}\  x\in\R,
$$
and there exists an $M_2>0$ such that
$$
\|u_n\|_{L^\infty}\le\frac{\|f(v_n)\|_{L^\infty}}{V_0}\le M_2\quad\mbox{for all } n\ge 1
$$
(see the beginning of proof of Lemma \ref{defi:2.2}). 
Now as in \eqref{eq:2.6} and \eqref{eq:2.7}, 
choose an $R_2>0$ so large that, for $|x|\ge R_2$ we have
$$
|f(v_n(x))|\le C_7 e^{-(1+{\eta_0})\eta_1|x|}$$
and
$$
-\Lambda\eta_1^2\left(1+\frac{\eta_0}{2}\right)^2+V_0
-C_7 e^{-\eta_0\eta_1 |x|/2}\ge \frac{V_0}{4}>0.
$$
For $R>R_2$,  set
$$
w_R(x):=M_3\left[e^{-\left(1+\frac{\eta_0}{2}\right)\eta_1(|x|-R_2)} 
+  e^{\left(1+\frac{\eta_0}{2}\right)\eta_1(|x|-R)} \right]
$$
where $M_3:=1+M_2.$ 
Since $w_R''=\left(1+\frac{\eta_0}{2}\right)^2\eta_1^2 w_R \geq 0$,  
as in the proof of Lemma \ref{defi:2.2}, for all $R_2\le |x|\le R$, we get
\begin{eqnarray*}
& &-\cM^\pm (w_R'') + V(x) w_R - f(v_n)\\
&\ge& -\cM^\pm (w_R'') + V(x) w_R - M_3f(v_n)\\
&\ge & \left\{ -\Lambda \left(1+\frac{\eta_0}{2}\right)^2\eta_1^2+V \right\} M_3 
\left[ e^{-\left(1+\frac{\eta_0}{2}\right)\eta_1(|x|-R_2)} 
+ e^{\left(1+\frac{\eta_0}{2}\right)\eta_1(|x|-R)} \right]\\
& & - M_3C_7e^{-\left(1+\eta_0 \right)\eta_1|x|}\\
&\ge &\left[  -\Lambda \left(1+\frac{\eta_0}{2}\right)^2\eta_1^2 
+ V_0-C_7e^{-\eta_0\eta_1R_2/2}  \right]M_3
e^{-\left(1+\frac{\eta_0}{2}\right)\eta_1|x|}\\
& & +\left[  -\Lambda \left(1+\frac{\eta_0}{2}\right)^2\eta_1^2 + V_0\right] 
e^{\left(1+\frac{\eta_0}{2}\right)\eta_1(|x|-R)}\\
&\ge & 0=-\cM^\pm (u_n'') + V(x) u_n - f(v_n).
\end{eqnarray*}
Noting
$$
0\le u_n(\pm R_2)\le M_2\le w_R(\pm R_2) \quad\mbox{and}\quad 
0\le u_n(\pm R)\le M_2\le w_R(\pm R),
$$
the comparison principle gives
$$
u_n(x)\le w_R(x)=M_3\left[ e^{-\left(1+\frac{\eta_0}{2}\right)\eta_1(|x|-R_2)} 
+ e^{\left(1+\frac{\eta_0}{2}\right)\eta_1(|x|-R)}  \right],
$$
for all $R_2\le |x|\le R$ and $n\ge 1.$
Letting $R\to\infty,$ we obtain
$$
u_n(x)\le C_8 e^{-(1+\frac{\eta_0}{2})\eta_1|x|}  
\quad \text{for all $x \in \R$ and $n \geq 1$.}
$$
Using this exponential decay and the equation, we observe that there exists $C_9>0$ such that
$$
\|u_n\|_{L^\infty}+\|u_n'\|_{L^\infty}+\|u_n''\|_{L^\infty} \leq C_9 
\quad \text{for all $n \geq 1$}.
$$
Thus, there exists $(u_{n_k})$ such that $u_{n_k}\to u_0$ in $C^2_{\rm loc}(\R)$, 
where $u_0$ satisfies
$$
u_0(x)\le C_8 e^{-\left(1+\frac{\eta_0}{2}\right)\eta_1|x|} \quad 
\text{for all } x \in \R.
$$
This implies that $u_0 \in X_{\eta_1}$ and 
$u_{n_k}\to u_0$ in $X_{\eta_1}$. Hence, $(u_{n})$ is relatively compact in $X_{\eta_1}$.

	Finally, we prove the continuity of $\mathcal{L}^\pm$. 
If $v_n\to v_0$ in $X_{\eta_1}$, then arguing as in the above, 
there exists a subsequence $(u_{n_k})$ such that 
$u_{n_k} \to u_0$ in $X_{\eta_1} \cap C^2_{\rm loc}(\R)$ where $u_0$ satisfies 
$$
-\cM^\pm (u_0'') + V(x) u_0 = f(v_0)  \quad\mbox{in }\,\R.
$$
By Lemma \ref{defi:2.2}, $u_0$  is uniquely determined and 
does not depend on choices of subsequences. 
Therefore, it is easily seen that the whole sequence $(u_n)$ converges to $u_0$ 
in $X_{\eta_1}$ and the maps ${\mathcal L}^\pm$ are continuous. 
\end{proof}

Using $\mathcal{L}^\pm$, the fact $f(s) \geq 0$ for every $s \in \R$ and 
the strong maximum principle, 
we notice that $u \in X_{\eta_1}$ is a solution of (\ref{eq:1.1}) if and only if 
$u={\mathcal L}^\pm(u)$ with $u \neq 0$.
%, and also $u \in X_{\eta_1}$ a solution of (\ref{eq:1.2}) 
%if and only if $u={\mathcal L}^-(u)$ with $u \neq 0$. 

		Next, in order to find a nontrivial fixed point of $\mathcal{L}^\pm$ in $X_{\eta_1}$, 
following the idea in \cite{dFLN-82} (cf. \cite{FQ-04}), 
we shall show that
\begin{enumerate}
	\item[i)] 
		There exists an $r_0>0$ such that 
		${\rm deg}_{X_{\eta_1}}({\rm id}-{\mathcal L}^\pm, B_{r_0}(0), 0)=1$.
	\item[ii)] 
		There exists an $r_1>r_0$ such that 
		${\rm deg}_{X_{\eta_1}}({\rm id} - {\mathcal L}^\pm, B_{r_1}(0), 0)=0$.
\end{enumerate}
Here ${\rm deg}_{X_{\eta_1}}({\rm id}-{\mathcal L}^\pm, \Omega, 0)$ 
stands for the degree of the map ${\rm id}-{\mathcal L}^\pm$ in $X_{\eta_1}$. 
From i) and ii), it follows that 
$$
{\rm id} - \mathcal{L}^\pm \neq 0 \quad 
{\rm on} \ \partial A_{r_0,r_1} \quad {\rm and} \quad 
{\rm deg}_{X_{\eta_1}}({\rm id}-{\mathcal L}^\pm, A_{r_0,r_1}, 0)=-1
$$
where $A_{r_0,r_1}:=\{u\in X_{\eta_1} \,|\, r_0 < \|u\|_{\eta_1}< r_1 \}.$
Thus, if we can prove i) and ii) we can find a solution of \eqref{eq:1.1} 
%and \eqref{eq:1.2} 
in $A_{r_0,r_1}$.

		First we show i), namely, 
\begin{lema}\label{defi:2.5}\hspace{-0,3cm} There exists an $r_0>0$ such that 
${\rm deg}_{X_{\eta_1}}( {\rm id} -{\mathcal L}^\pm, B_{r_0}(0), 0)=1.$
\end{lema}
\begin{proof}
It suffices to prove that there exists an $r_0>0$ such that 
$({\rm id}-\beta{\mathcal L}^\pm)(u)\not =0$ 
for all $u\in \partial B_{r_0}(0)$ and all $\beta\in [0,1]$ since 
the homotopy invariance gives 
$$
{\rm deg}_{X_{\eta_1}}({\rm id}-{\mathcal L}^\pm, B_{r_0}(0), 0) 
= {\rm deg}_{X_{\eta_1}}({\rm id}, B_{r_0}(0), 0)=1. 
$$
We first notice that for $\beta>0$, the equations 
$u=\beta{\mathcal L}^\pm(u)$ are equivalent to 
$$
-\cM^\pm (u'') + V(x) u = \beta f(u)  \quad\mbox{in }\,\R
$$
for $u\in X_{\eta_1}$. If $u\in X_{\eta_1}\setminus \{0\}$ satisfies 
$u=\beta{\mathcal L}^\pm(u)$  with $\beta>0$, 
then the fact that $f(s)\geq 0$ for all $s \in \R$ yields $u>0$ in $\R$. 
Since $u(x)\to 0$ as $|x|\to\infty$, let $x_0\in\R$ be a maximum point of $u$. 
As in the proof of Lemma \ref{defi:2.2}, 
from $\beta \in [0,1]$ and $f(s)>0$ for $s>0$ due to (f4), we get 
$$
0<V_0\le V(x_0) \leq \frac{\beta f(u(x_0))}{u(x_0)} \leq \frac{f(u(x_0))}{u(x_0)}.
$$
By (f2), we may find a $\delta_1>0$, which is independent of $\beta$ and $u$, 
so that
$$
\delta_1\le u(x_0)=\|u\|_{L^\infty}\le \|u\|_{\eta_1}
$$
for all $u\in X_{\eta_1}\setminus \{0\}$ and $\beta\in (0,1]$ with 
$u = \beta  \mathcal{L}^\pm (u)$. 
Therefore, selecting an $r_0 \in (0,\delta_1)$, we see that
$$
({\rm id}-\beta{\mathcal L}^\pm )(u)\not= 0
$$
for all $u\in\partial B_{r_0}(0)$ and for all $\beta\in (0,1]$. Thus the lemma holds.
\end{proof}

		To show ii), we need some preparations. 
From (V3), we may select a $\kappa_0 > 0$ so that 
	\begin{equation}\label{eq:2.9}
		[-3 \kappa_0, 3 \kappa_0] \subset 
		[ V_\infty - V > 0] := \{ x \in \R  \ |\ V_\infty - V(x) >0 \}. 
	\end{equation}
Next choose a $\varphi_0 \in C^\infty_0(\R)$ satisfying  
	\begin{equation}\label{eq:2.10}
		\begin{aligned}
			& \varphi_0(-x) = \varphi_0(x), \quad 
			0 \leq \varphi_0 \leq 1 \quad {\rm in}\ \R, 
			\quad \varphi'_0(x) \leq 0 \quad 
			{\rm in} \ [0,\infty), 
			\\
			& \varphi_0(x) = 1 \quad 
			{\rm if}\ 0 \leq x \leq \kappa_0, \quad 
			\varphi_0(x) = 0 \quad {\rm if\ }
			2\kappa_0 \leq x .
		\end{aligned}
	\end{equation}
Then we first prove
	\begin{lema}\label{defi:2.6}
		There exists a $\tilde{t} = \tilde{t}(f,V_\infty) > 0$ such that 
			\begin{equation}\label{eq:2.11}
				\frac{\kappa_0^2}{4 \Lambda} t 
				\leq \| u \|_{L^\infty([-\kappa_0,\kappa_0])} 
				\leq \| u \|_{X_{\eta_1}} 
				\quad 
				{\rm for\ each} \ t \geq \tilde{t} \quad {\rm and} \quad 
				u \in \cS_t^\pm
			\end{equation}
		where 
			\[
				\begin{aligned}
					\cS_t^\pm &:= \{ u \in X_{\eta_1}\ |\ 
					-\cM^\pm (u'') + V(x) u = f(u) + t \varphi_0  \}.
%					\\
%					\cS_t^- &:= \{ u \in X_{\eta_1}\ |\ 
%						-\cM^- (u'') + V(x) u = f(u) + t \varphi_0  \}. 
				\end{aligned}
			\]
	\end{lema}

	\begin{proof}
		By (f2) and (f3), there exists a $c(f,V_\infty) > 0$ such that 
	\[
		\inf_{0 \leq s} \left( \frac{f(s)}{s} - V_\infty \right) s \geq 
		- c(f,V_\infty).
	\]
Choose a $\tilde{t} = \tilde{t}(f,V_\infty)>0$ so that 
if $t \geq \tilde{t}$, then 
$-c(f,V_\infty) + t \geq t /2$. For this $\tilde{t}$, 
we shall prove that \eqref{eq:2.11} holds.

		Let $t \geq \tilde{t}$ and $u \in \cS_t^\pm$. 
Since $t>0$, we have $u \not\equiv 0$. 
Thus $u > 0$ in $\R$ due to $f(s) \geq 0$ in $\R$ and 
the strong maximum principle. Hence, (V3) yields 
	\[
		\begin{aligned}
		- \cM^\pm (u'') &= f(u) + t \varphi_0 - V(x) u 
		= \left( \frac{f(u)}{u} - V(x) \right) u + t \varphi_0 
		\\
		& \geq \left( \frac{f(u)}{u} - V_\infty \right) u + t \varphi_0 
		\geq - c(f,V_\infty) + t \varphi_0 \quad {\rm in}\ \R.
		\end{aligned}
	\]
By the definition of $\varphi_0$, we see
	\[
		- \cM^\pm (u'') \geq \frac{t}{2} \quad {\rm in} \ 
		[-\kappa_0,\kappa_0],
	\]
which implies 
	\[
		u'' \leq - \frac{t}{2 \Lambda} \quad {\rm in}\ [-\kappa_0,\kappa_0]. 
	\]
Integrating the inequality over $[x,y] \subset [-\kappa_0, \kappa_0]$, one has 
	\begin{equation}\label{eq:2.12}
		u'(y) \leq u'(x) - \frac{t}{2 \Lambda} (y -x) 
		\quad {\rm for}\ -\kappa_0 \leq x \leq y \leq \kappa_0.
	\end{equation}

		Now we divide our arguments into two cases: 

\smallskip

{\bf Case 1} There exists an $x_0 \in [-\kappa_0,0]$ such that 
$u'(x_0) \leq 0$. 

{\bf Case 2} $u'>0$ in $[-\kappa_0,0]$. 

\smallskip

		In Case 1, we put $x=x_0$ and integrate \eqref{eq:2.12} in $y$ 
over $[x_0,\kappa_0]$ to obtain 
	\[
		u(\kappa_0) \leq u(x_0) + u'(x_0) (\kappa_0 - x_0) 
		- \frac{t}{4 \Lambda} (\kappa_0 - x_0)^2 
		\leq u(x_0) - \frac{t}{4 \Lambda} (\kappa_0 - x_0)^2. 
	\]
Hence, 
	\[
		\begin{aligned}
			\| u \|_{L^\infty ([-\kappa_0,\kappa_0]) } 
			& \geq u(x_0) \geq u(\kappa_0) + \frac{t}{4\Lambda} (\kappa_0-x_0)^2 
			\geq \frac{t}{4\Lambda} (\kappa_0-x_0)^2 
			\geq  \frac{\kappa_0^2 }{4 \Lambda} t. 
		\end{aligned}
	\]
Thus \eqref{eq:2.11} holds.

		In Case 2, putting $y=0$ in \eqref{eq:2.12}, 
it follows that 
	\[
		- \frac{t}{2 \Lambda} x \leq u'(x) \quad {\rm for\ every}\ x \in [-\kappa_0,0]. 
	\]
Integrating this inequality over $[-\kappa_0,0]$, we obtain 
	\[
		\frac{\kappa_0^2}{4 \Lambda} t \leq u(0) - u(-\kappa_0) 
		< u(0) \leq \| u \|_{L^\infty([-\kappa_0,\kappa_0])}.
	\]
Thus \eqref{eq:2.11} holds and we complete the proof. 
	\end{proof}

Next, we shall prove some properties of elements in $\cS_t^\pm$.

	\begin{lema}\label{defi:2.7}
		Let $ t \geq 0$ and $u \in \cS_t^\pm \setminus \{0\}$. Then either 
			
	\emph{(i)} 
		There exists an $x_0 \in \R$ such that 
		$u'(y) < 0 < u'(x) $ for all $x < x_0 < y$
	
	or else 
	
	\emph{(ii)} 
		There are $y_0 < 0 < z_0$ such that 
		$u'(y_0) = 0 = u'(z_0)$ and 
		$u'(x) \neq 0$ if $x \neq y_0, z_0$. 
	
	\noindent
	In particular, every $u \in \cS_t^\pm \setminus \{0\}$ 
	has only one maximum point in $\R$. 
	\end{lema}

\begin{proof}

For $u \in \cS_t^\pm \setminus \{0\}$, 
it suffices to prove the following claim: 

\medskip

\noindent
{\bf Claim:} {\sl If $u'(z_0) = 0$ holds for some $z_0 \geq 0$, then 
$u'(x) < 0$ for every $x > z_0$. Similarly, if $u'(y_0) = 0$ holds for $y_0 \leq 0$, 
then $u'(x) >0$ for all $x < y_0$. In particular, 
each $u \in \cS_t^\pm \setminus \{0\}$ has at most one critical point in 
$[0,\infty]$ (resp. $(-\infty,0]$).}

\medskip

		We first remark that since $u(-x)$ satisfies 
the same type of equation by \eqref{eq:2.10} and (V1)--(V3), 
it is enough to prove the first assertion. To this end, suppose that 
$z_0 \geq 0$ satisfies $u'(z_0)=0$ and set, for all $x\in \R$, 
\[
	\tilde u(x)=u(z_0+|x|),\quad \tilde V(x)=V(z_0+|x|) 
	\quad\mbox{and}\quad  \tilde \varphi_0(x)=\varphi_0(z_0+|x|).
\]
Then, since $u'(z_0)=0$ and $z_0\ge 0$, $\tilde u\in C^1(\R)\cap C^2(\R\setminus \{0\})$, 
$\tilde u(x)\to 0$ as $|x|\to \infty$. 
Moreover $\tilde u$, $\tilde V$ and $\tilde \varphi_0$ are even and
$\tilde V'(x)\ge 0$, $\tilde \varphi_0'(x)\le 0 $ for a.a. $x \ge 0$ and 
$$
-\cM^{\pm}(\tilde u'')+\tilde V \tilde u=f(\tilde u)+t\tilde \varphi_0\quad\mbox{in}\quad \R\setminus\{0\}.
$$
Furthermore, by the differential equations and $u\in C^2(\R)$, 
we have $\tilde u\in C^2(\R)$ and the equation above is satisfied in $\R$.

	We shall prove Claim by the moving plane method. 
For $\lambda>0$, 
define $x_\lambda=2\lambda-x$, $\Sigma_\lambda=\{x\in(0,\infty)\,|\, \lambda<x\}$ and
$$
u_\lambda(x)=\tilde u(x_\lambda)-\tilde u(x),\quad \varphi_\lambda(x)= \tilde \varphi_0(x_\lambda)-\tilde \varphi_0(x).
$$
Since
\begin{eqnarray*}
-\cM^{\pm}(\tilde u'')(x_\lambda)+\tilde V(x_\lambda) \tilde u(x_\lambda)&=&f(\tilde u(x_\lambda))+t\tilde \varphi_0(x_\lambda)\\
-\cM^{\pm}(\tilde u'')(x)+\tilde V(x) \tilde u(x) &=&f(\tilde u(x))+t\tilde \varphi_0(x),
\end{eqnarray*}
we have
\begin{eqnarray*}
&&-\left(\cM^{\pm}(\tilde u'')(x_\lambda)-\cM^{\pm}(\tilde u'')(x)\right) 
+ \left(\tilde V(x_\lambda)-\tilde V(x) \right) \tilde u(x_\lambda) 
+\tilde V(x)u_\lambda
\\
&=& f(\tilde u(x_\lambda))-f(\tilde u(x))+t(\tilde \varphi_0(x_\lambda)-\tilde \varphi_0(x)).
\end{eqnarray*}
Noting \eqref{eq:2.8}, $|x_\lambda|\le |x|$, $\tilde V(x_\lambda)\le \tilde V(x)$ and 
$\tilde{\varphi}_0(x) \le \tilde{\varphi}_0(x_\lambda)$ for all $x\in\Sigma_\lambda$, we have
$$
-\cM^-(u_\lambda'')+\tilde V(x)u_\lambda\ge f(\tilde u(x_\lambda))-f(\tilde u(x))\quad\mbox{in}\quad \Sigma_\lambda.
$$
Moreover, from
\begin{eqnarray*}
f(\tilde u(x_\lambda))-f(\tilde u(x))%&=&\int_0^1\frac{\rd}{\rd \theta}f(\tilde u(x)+\theta u_\lambda(x)) \rd \theta\\
&=&\int_0^1f'(\tilde u(x)+\theta u_\lambda(x))\rd\theta u_\lambda(x)\\
&=:&g_\lambda(x)u_\lambda(x),
\end{eqnarray*}
we have
$$
-\cM^-(u_\lambda'')+(\tilde V(x)-g_\lambda(x))u_\lambda(x) \ge 0 \quad\mbox{in}\quad \Sigma_\lambda.
$$
Since $\tilde u>0$ in $\R$, $f'(0) = 0$ by (f2) and $\tilde u(x)\to 0$ as $x\to\infty$, 
it is not difficult to see that 
the strong maximum principle implies that 
for all $\lambda$ sufficiently large,
\begin{equation}\label{eq:2.13}
u_\lambda>0\quad\mbox{in}\quad  \Sigma_\lambda, \quad 
u_\lambda'(\lambda)=-2\tilde u'(\lambda) = -2 u'(z_0+\lambda)>0. 
\end{equation}

	Next, set
$$
\lambda_*=\inf\{\lambda>0\,|\, u_{\tilde \lambda}>0 \quad\mbox{in}\quad 
\Sigma_{\tilde \lambda}\quad\mbox{for all}\quad \tilde{\lambda}>\lambda\}.
$$
From the above observation, we have $0 \leq \lambda_* < \infty$. 
In addition, notice that if $u_\lambda \geq 0$ in $\Sigma_\lambda$, then 
\begin{eqnarray*}
0&\le&-\cM^-(u_{\lambda}'')+\left(\tilde V(x)-g_{\lambda}(x)\right)u_{\lambda}(x)\\
& \le&  
-\cM^-(u_{\lambda}'')+\left(\tilde V(x)-g_{\lambda}(x)\right)_+u_{\lambda}(x)
\quad\mbox{in}\quad \Sigma_{\lambda}.
\end{eqnarray*}
In particular, since $u_{\lambda_*} \geq 0$ in $\Sigma_{\lambda_\ast}$, 
the strong maximum principle yields either
\begin{enumerate}
\item[(i)] $u_{\lambda_*}>0$ in $\Sigma_{\lambda_*},$ $u_{\lambda_*}'(\lambda_*)>0$

or else
\item[(ii)] $u_{\lambda_*}\equiv 0$ in $\Sigma_{\lambda_*}$.
\end{enumerate}

Next we prove
that if $\mu>0$ and $u_\mu>0$ in $\Sigma_\mu$ hold, 
then there exists an $\epsilon_\mu > 0$ such that
$u_{\tilde \mu} > 0$ in $\Sigma_{\tilde \mu}$ provided $|\mu-\tilde \mu|< \epsilon_\mu$.
To see this, we remark that
$$
u_{\tilde \mu}\to u_{ \mu}\quad\mbox{in}\quad C^1_{\rm loc}([\mu,\infty)).
$$
Since $u_{\mu}'(\mu) > 0$ holds due to $u_\mu > 0$ in $\Sigma_\mu$ 
and the strong maximum principle, 
for sufficiently small $\epsilon_\mu$, we observe that
$|\mu-\tilde \mu|<\epsilon_\mu$ implies $u_{\tilde \mu}>0$ in $(\tilde \mu, R_\epsilon)$ where
$R_\epsilon>0$ is chosen so that $x\ge R_\epsilon$ implies $g_{\tilde\mu}(x)\le V_0/2$.

From
$$
-\cM^-(u_{\tilde \mu}'')+(V-V_0/2)u_{\tilde\mu}\ge 0\quad\mbox{in}\quad [R_\epsilon,\infty),
\quad u_{\tilde \mu}(x) \to 0 \quad {\rm as} \ x \to \infty
$$
and the strong maximum principle, 
$u_{\tilde \mu}$ cannot take a non-positive minimum. Hence
$|\mu-\tilde \mu|<\epsilon_\mu$ implies $u_{\tilde \mu}>0$ in $\Sigma_{\tilde\mu}$ and $u_{\tilde \mu}'(\tilde\mu)>0$.

By this claim we see that if $u_{\lambda_*}>0$ in $\Sigma_{\lambda_*}$, 
then $\lambda_*=0.$ Thus, $\lambda_*=0$ holds provided (i) occurs. 
Moreover, we also see from \eqref{eq:2.13} that $\tilde u'(x) < 0$ for all $x > 0$.

On the other hand, let us consider the case  $\lambda_*>0$ and $u_{\lambda_*}\equiv 0$ in $\Sigma_{\lambda_*}$. In this case, we notice that
$-2\tilde u'(\lambda)=u_\lambda'(\lambda)>0$ for all $\lambda>\lambda_*$ and 
$\tilde u(2\lambda_*-x)=\tilde u(x)$ for all $x\ge \lambda_*$.
Since $\tilde u'(0)=0$, we have $\tilde u'(2\lambda_*)=0,$ which is a contradiction. 
Hence, (ii) only occurs when $\lambda_*=0$ and it follows from \eqref{eq:2.13} that 
$\tilde u'(x) < 0$ for all $x > 0$. 

By the above observations, we obtain $\lambda_*=0$ and $\tilde u'(x)<0$ for all $x>0$, 
which implies $u'(x)<0$ for all $x>z_0$. Thus we complete the proof.
 \end{proof}

	\begin{lema}\label{defi:2.8}
		There exists an $M_0 > 0$ such that 
			\[
				\| u \|_{L^\infty(\R)} \leq M_0 \quad 
				{\rm for \ each} \ u \in \cS_t^\pm \quad {\rm and} \quad 
				t \geq 0.
			\]
	\end{lema}

	\begin{proof}
We argue by contradiction and suppose that 
there are $(s_n) \subset [0,\infty)$ and 
$u_n \in \cS_{s_n}^\pm$ such that $\tau_n := \| u_n \|_{L^\infty(\R)} \to \infty$. 
Thanks to Lemma \ref{defi:2.7}, 
let $(x_n) \subset \R$ be a unique maximum point of $(u_n)$ and 
set 
	\[
	\begin{aligned}
		v_n(x) &:= \frac{1}{\tau_n} u_n \left( x_n + \sqrt{\frac{\tau_n}{f(\tau_n)}} x \right),
		\quad
		\varphi_n(x) := \varphi_0\left( x_n + \sqrt{\frac{\tau_n}{f(\tau_n)}} x \right), \\
		V_n(x) &:= V\left( x_n + \sqrt{\frac{\tau_n}{f(\tau_n)}} x \right).
	\end{aligned}
	\]
Then $v_n$ satisfies 
	\[
		v_n(x) \leq v_n(0)=1 , \quad 
		- \cM^\pm (v_n'') + \frac{\tau_n}{f(\tau_n)} V_n v_n 
		= \frac{f(\tau_n v_n)}{f(\tau_n)} + \frac{s_n}{f(\tau_n)} \varphi_n
		\quad {\rm in}\ \R.
	\]
Recalling Lemma \ref{defi:2.6}, we have  
	\[
		\frac{\kappa_0^2}{4 \Lambda} s_n 
		- \frac{\kappa_0^2}{4 \Lambda} \tilde{t}
		\leq \| u_n \|_{L^\infty(\R)} = \tau_n 
		\quad {\rm for\ all} \ n.
	\]
Hence, by (f3), $s_n / f(\tau_n) \to 0$ as $n \to \infty$. 
Moreover, noting that $f(s)$ is increasing in $[0,\infty)$ by (f4), 
it follows from $v_n (x) \leq 1$ that 
	\[
		\left| \frac{f(\tau_n v_n)}{f(\tau_n)} \right| 
		\leq 1 \quad {\rm in} \ \R.
	\]
Noting $v_n'(0) = 0$ and $\tau_n / f(\tau_n) \to 0$, 
we may extract a subsequence (still denoted by $(n)$) such that 
	\[
		v_n \to v_0 \quad {\rm \in} \ C^1_{\rm loc}(\R), \quad 
		0 \leq v_0 \leq 1 \quad {\rm in} \ \R, \quad 
		v_0(0)= 1, \quad v_0'(0) = 0.
	\]
Furthermore, by $0 \leq v_n \leq 1$ and (f3), we have 
	\[
		\frac{f(\tau_n v_n)}{f(\tau_n)}
		 \to \bar{f}(v_0)\quad 
		{\rm in} \ C_{\rm loc}([ v_0 > 0 ]). 
	\]
Since $0 \in [v_0>0]$, let $I$ be a component of $[v_0>0]$ 
satisfying $0 \in I$. Then we have 
	\[
		v_n \to v_0 \quad {\rm in}\ C^2_{\rm loc}(I), \quad 
		- \cM^\pm (v_0'') = \bar{f}(v_0) \quad {\rm in}\ I.
	\]
When $I=(-c_2,c_1)$ and $c_1 < \infty$, since $v_0'(0)=0$, 
$v_0(0) = 1$, $\bar{f} \geq 0$, $\bar{f}(1) = 1 > 0$ 
and $v_0(c_1) = 0$, 
we observe that $v_0'(c_1) < 0$, however, 
this contradicts $ 0 \leq v_0 \leq 1$ in $\R$.

		On the other hand, if $I=(-c_2,\infty)$, then 
by $v_0'(0)=0$, $v_0(0)=1$ and $-\cM^\pm(v_0'') =  \bar{f}(v_0)$, 
we observe that $v_0$ must hit a zero at some $x_0>0$ with 
$v_0'(x_0) < 0$, however this contradicts $0 \leq v_0 \leq 1$ again. 
Hence, Lemma \ref{defi:2.8} holds and  we complete the proof. 
	\end{proof}

The next proposition is a key in order to prove  
${\rm deg}_{X_{\eta_1}}( {\rm id} - \mathcal{L}^\pm , B_{r_1}(0), 0  ) = 0$ 
for some $r_1>r_0$.

	\begin{prop}\label{defi:2.9}
		There exists an $M_1>0$ such that 
			\[
				\| u \|_{X_{\eta_1}} \leq M_1
				\quad {\rm for\ each} \ 
				u \in \cS_t^\pm \quad {\rm and} \quad 
				t \geq 0.
			\]
	\end{prop}

Assuming Proposition \ref{defi:2.9}, we first prove Theorem \ref{defi:1.2}. 
Before the proof, we remark that 
for every $t \geq 0$ and $v \in X_{\eta_1}$, the equations
	\[
		- \cM^\pm(u'') + V(x) u 
		= f(v) + t \varphi_0 \quad {\rm in} \ \R
	\]
have unique solutions in $X_{\eta_1}$. Indeed, we may prove this claim 
in a similar way to the proof of Lemma \ref{defi:2.2} thanks to 
$\varphi_0 \in C^\infty_0(\R)$. 
Thus, we denote by $F^\pm(t,v)$ these unique solutions. 
Furthermore, we may show that the maps 
$(t,v) \mapsto F^\pm(t,v) : [0,\infty) \times X_{\eta_1} \to X_{\eta_1}$ 
are compact as in Lemma \ref{defi:2.4}.

	Now we prove Theorem \ref{defi:1.2}.

	\begin{proof}[Proof of Theorem \ref{defi:1.2}]
Choose $r_1 := M_1+\frac{\kappa_0^2}{4\Lambda} \tilde{t} + r_0$ where 
$M_1$, $\tilde{t}$ and $r_0$ appear in Proposition \ref{defi:2.9} and 
Lemmas \ref{defi:2.6} and \ref{defi:2.5}. We first claim that 
	\begin{equation}\label{eq:2.14}
		u - F^\pm (t_1, u) \neq 0 \quad {\rm in}\ 
		\overline{B_{r_1}}
	\end{equation}
where $t_1>\tilde{t}$ is chosen so that $\frac{\kappa_0^2}{4\Lambda} t_1 > r_1$. 
Indeed, let $u \in \overline{B_{r_1}}$ satisfy $ u - F^\pm(t_1,u) = 0$. 
Noting $u \in \cS^\pm_{t_1}$ and $t_1 > \tilde{t}$, 
Lemma \ref{defi:2.6} gives a contradiction:
	\[
		r_1 \geq \| u \|_{X_{\eta_1}} \geq \| u \|_{L^\infty(\R)} 
		\geq \frac{\kappa_0^2}{4\Lambda}t_1 > r_1. 
	\]
Hence, \eqref{eq:2.14} holds.

	Since Proposition \ref{defi:2.9} and the choice of $r_1$ imply
	\[
		u - F^\pm(t,u) \neq 0 \quad {\rm on} \ \partial B_{r_1} 
		\quad {\rm for\ every} \ t \geq 0,
	\]
it is easily seen from \eqref{eq:2.14} and the homotopy invariance of degree that 
	\[
		{\rm deg}_{X_{\eta_1}}( {\rm id} - \mathcal{L}^\pm, B_{r_1}(0), 0  ) = 0.
	\] 
Combining this with Lemma \ref{defi:2.5}, we obtain 
	\[
		{\rm deg}_{X_{\eta_1}} ( {\rm id} - \mathcal{L}^\pm, A_{r_0,r_1}, 0 ) = -1
	\]
and solutions of \eqref{eq:1.1} %and \eqref{eq:1.2} 
in $A_{r_0,r_1}$. 
This completes the proof. 
	\end{proof}

		Before proceeding to the proof of Proposition \ref{defi:2.9}, 
we remark the following fact on the function 
$g_\infty(s) := f(s) - V_\infty s$, which will be used below. 

\medskip

\noindent
{\bf Fact:} {\sl There exists a unique $s_\infty >0$ such that 
	\begin{equation}\label{eq:2.15}
		g_\infty(s) < 0 = g_\infty(s_\infty) < g_\infty(t) 
		\quad \text{for all $ 0 < s < s_\infty < t$}. 
	\end{equation}
}

\medskip 

This fact follows from (f1)--(f4). In fact, for sufficiently small $s > 0$, 
by (f2), we get $g_\infty(s) < 0$. On the other hand, (f3) yields 
$g_\infty(s) \to \infty$ as $s \to \infty$, hence, 
there exists an $s_\infty > 0$ so that $g_\infty(s_\infty) = 0$. 
Moreover, from 
	\[
		g_\infty(s) = s \left( \frac{f(s)}{s} - V_\infty \right)
	\]
and (f4), we see that \eqref{eq:2.15} holds. 

\medskip

Now we prove Proposition \ref{defi:2.9}. 

	\begin{proof}[Proof of Proposition \ref{defi:2.9}]
We argue indirectly and suppose that 
there exists 
$(s_n,u_n) \in [0,\infty) \times X_{\eta_1}$ such that 
$u_n \in \cS^\pm_{s_n}$ and 
$ \| u_n \|_{X_{\eta_1}} \to \infty$. Remark that 
$u_n$ satisfies 
	\[
		- \cM^\pm (u_n'') + V u_n = f(u_n) 
		+ s_n \varphi_0 \quad {\rm in}\ \R.
	\]
By Lemma \ref{defi:2.7}, $u_n$ has only one maximum point and 
denote it by $x_n$. Our first aim is to show 
	\begin{equation}\label{eq:2.16}
		\text{$(x_n)$ is bounded.}
	\end{equation}

		To prove \eqref{eq:2.16}, suppose that $x_n \to \infty$. We may assume 
$3 \kappa_0 < x_n$. Setting
	\[
		v_n(x) := u_n(x + x_n), \quad 
		V_n(x) := V(x+ x_n), \quad 
		\varphi_n(x) := \varphi_0(x + x_n),
	\]
we see $\varphi_n \equiv 0$ in $[0,\infty)$ thanks to $3 \kappa_0 < x_n$. 
Furthermore, by Lemma \ref{defi:2.7}, we have 
	\[
		\begin{aligned}
			&-\cM^\pm (v_n'') + V_n v_n = f(v_n) + s_n \varphi_n \quad 
			{\rm in} \ \R, \quad v_n(0) = \max_{\R} v_n > 0,
			\\
			& v_n'(y) \leq 0 \leq v_n'(x) \quad 
			{\rm for}\ x < 0 < y, \quad V_n \to V_\infty, \ \ s_n \varphi_n \to 0 \quad 
			{\rm in \ } C_{\rm loc}(\R). 
		\end{aligned}
	\]
In the sequel, we divide our arguments into several steps. 

\medskip

\noindent
{\bf Step 1:} {\sl One has 
	\begin{equation}\label{eq:2.17}
		v_n \to \omega_{\pm} \quad {\rm strongly \ in}\ C^2_{\rm loc}(\R)
	\end{equation}
where $\omega_\pm$ are unique solutions of 
\eqref{eq:2.1} and \eqref{eq:2.2} (see Proposition \ref{defi:2.1}). 
}

\medskip

		We first notice that $(v_n)$ is bounded in $L^\infty(\R)$ 
due to Lemma \ref{defi:2.8}. Combining with $V_n \to V_\infty$ and 
$s_n \varphi_n \to 0$ in $C_{\rm loc}(\R)$, 
we may extract a subsequence (still denoted by $(n)$) so that 
	\[
		\begin{aligned}
			& v_n \to v_0 \quad {\rm in}\ C^2_{\rm loc}(\R), \quad 
			- \cM^\pm (v_0'') + V_\infty v_0 = f(v_0) \quad {\rm in}\ \R, 
			\\
			& v_0(0)=\max_{\R} v_0, \quad 
			0 \leq v_0 \quad {\rm in}\ \R, \quad 
			v_0'(y) \leq 0 \leq v_0'(x) \quad 
			{\rm for\ } x < 0 < y. 
		\end{aligned}
	\]
By $v_n(0) = \max_{\R} v_n$, we have $v_n''(0) \leq 0$. 
Since $v_n(0) > 0$ and $\varphi_n \equiv 0$ on $[0,\infty)$, we get 
	\[
		f(v_n(0)) = - \cM^\pm (v_n''(0)) + V_n(0) v_n(0) 
		\geq V_n(0) v_n(0),
	\]
which implies 
	\[
		V_n(0) \leq \frac{f(v_n(0))}{v_n(0)}. 
	\]
By $V_n(0) \to V_\infty$ and (f2), we may find a $\delta_0 > 0$ so that 
$v_n(0) \geq \delta_0$ for all $n$. Thus $v_0(0) \geq \delta_0$ and 
$v_0 > 0$ in $\R$. 
Now from $v_0'(y) \leq 0$ in $[0,\infty)$, one has 
	\[
		v_{0,\infty} := \lim_{x \to \infty} v_0(y) \geq 0.
	\]
Since $-\cM^\pm (v_0'') = g_\infty(v_0)$ in $\R$, it follows from \eqref{eq:2.15} that 
	\[
		\text{either $v_{0,\infty} = 0$ or $v_{0,\infty} = s_\infty > 0$.}
	\]
If $v_{0,\infty} = 0$, then by Proposition \ref{defi:2.1}, we have $v_0 = \omega_\pm$ 
and Step 1 holds.

		Now we assume $v_{0,\infty} = s_\infty$. 
By $v_0' \leq 0$ in $[0,\infty)$ and \eqref{eq:2.15}, 
we have $v_0 \geq s_\infty$ in $[0,\infty)$ and 
	\[
		- \cM^\pm(v_0'') = g_\infty(v_0) \geq 0 \quad {\rm on}\ [0,\infty). 
	\]
Moreover, if $v_0(0) > s_\infty$, then the strict inequality holds at $x=0$. 
However, this contradicts facts $v_0'(0)=0 > v_0''(0)$, 
$v_0''(x) \leq 0$ for $x \in [0,\infty)$ and 
$v_0(x) \to s_\infty$ as $x\to \infty$. 
Thus we get $v_0 \equiv s_\infty$ in $[0,\infty)$.

	Next, we put 
	\[
		\begin{aligned}
			&
			E_{n,+}(x) := \frac{\Lambda}{2} (v_n'(x))^2 
			+ F(v_n(x)) - \frac{V_n}{2} v_n^2 
			& &\text{for $\cM^+$},
			\\
			&
			E_{n,-}(x) := \frac{\lambda}{2} (v_n'(x))^2 
			+ F(v_n(x)) - \frac{V_n}{2} v_n^2
			& &\text{for $\cM^-$}. 
		\end{aligned}
	\]
We also put $h_n(x) := V_n(x) - f(v_n(x)) / v_n(x)$. 
Recalling $V_n(x) = V(x+x_n)$ and $x_n \to \infty$, 
we may assume that $V_n'(x) \geq 0$ in $(0,\infty)$. Notice also that 
$v_n$ is strictly decreasing in $(0,\infty)$ by Lemma \ref{defi:2.7}. 
Hence (f4) yields that 
$h_n(x)$ is strictly increasing in $[0,\infty)$. Since 
$v_n''(0) \leq 0$ and $\cM^\pm (v_n'') = v_n h_n$ in $[0,\infty)$, 
we see $h_n(0) \leq 0$. Noting $h_n(x) \to V_\infty > 0$ as $x \to \infty$, 
there exists a unique $z_n^\pm \geq 0$ such that $h_n(z_n^\pm) = 0$. 
Therefore, one has 
	\[
		v_n''(x) < 0 < v_n''(y) \quad 
		{\rm for}\ 0 \leq x < z_n^\pm < y.
	\]
Moreover, taking a subsequence if necessary, we may assume 
$v_n(z_n^\pm) \to \tilde{s} \geq 0$ since $v_n(z_n^\pm)$ is bounded. 
Noting $V_n(z_n^\pm) \to V_\infty$ as $n \to \infty$ and 
letting $n \to \infty$ in $h_n(z_n^\pm) = 0$, it follows from (f2) that 
	\[
		\tilde{s} > 0 \quad {\rm and} \quad 
		V_\infty = \frac{f(\tilde{s})}{\tilde{s}}.
	\]
Thus by \eqref{eq:2.15}, we obtain $\tilde{s} = s_\infty$ and 
$v_n(z_n^\pm) \to s_\infty$. Recalling $v_n''(x) \geq 0$ for $x \geq z_n^\pm$, 
$V_n' \geq 0$ in $[0,\infty)$ and 
$\varphi_n \equiv 0$ in $[0,\infty)$, we have 
	\[
		\begin{aligned}
			E_{n,+}'(x) &= v_n'(x) 
			\left( \Lambda v_n'' + f(v_n) - V_n v_n \right) 
			- \frac{V_n'}{2} v_n^2 
			= - \frac{V_n'}{2} v_n^2 \leq 0 \quad 
			{\rm in}\ [z_n^+,\infty), 
			\\
			E_{n,-}'(x) &= - \frac{V_n'}{2} v_n^2 \leq 0 
			\quad {\rm in} \ [z_n^-,\infty).
		\end{aligned}
	\]
Thanks to $E_{n,\pm}(x) \to 0$ as $x \to \infty$, one sees 
$E_{n,\pm}(z_n^\pm) \geq 0$. 
Since it follows from \eqref{eq:2.15} that 
	\[
		V_n(z_n^\pm) \to V_\infty, \quad 
		G_\infty (s_\infty) 
		= \int_0^{s_\infty} g_\infty(s) \rd s 
		= \min_{[0,\infty)} G(s) < 0,
	\]
we obtain
	\[
		(v_n'(z_n^\pm))^2 \geq 
		\frac{2}{\Lambda} 
		\left\{ \frac{V_n(z_n^\pm)}{2} v_n^2(z_n^\pm) - F(v_n(z_n^\pm))   \right\} 
		\to -\frac{2}{\Lambda} G_\infty(s_\infty) > 0
	\]
By the fact that $(v_n'')$ is bounded in $[-1,\infty)$, we may find a 
$\delta_1, \delta_2 >0$ so that 
	\[
		|v_n'(x)| \geq \delta_1 > 0 \quad 
		{\rm in} \ [z_{n}^\pm - \delta_2, z_n^\pm + \delta_2].
	\]
Due to this and the fact $v_n'(0) = 0$, shrinking $\delta_2>0$ if necessary, 
we may assume $z_n^\pm \geq \delta_2 > 0$ for any $n$. 
Furthermore, by $v_n(z_n^\pm) \to s_\infty$ and $v_n' \leq 0$ in $[0,\infty)$, 
we obtain 
	\[
		v_n(0) \geq v_n(z_n^\pm - \delta_2) 
		= v_{n} (z_n^\pm) - \int_{z_n^\pm - \delta_2}^{z_n^\pm} v_n'(x) \rd x 
		\geq v_n (z_n^\pm) + \delta_1 \delta_2 \to s_\infty + \delta_1 \delta_2.
	\]
However, this contradicts $v_0 \equiv s_\infty$ in $[0,\infty)$. 
Thus $v_{0,\infty} = 0$ and Step 1 holds.

\bigskip

		To proceed further, we need some preparations. 
First, combining the monotonicity of $v_n$ with \eqref{eq:2.17}, 
we can prove that 
	\begin{equation}\label{eq:2.18}
		v_n \to \omega_{\pm} \quad {\rm strongly\ in}\ L^\infty(\R). 
	\end{equation}
Moreover, by the differential equation, 
we also derive the uniform exponential decay at $x = \infty$:
	\begin{equation}\label{eq:2.19}
		v_n(x) + |v_n'(x) | \leq c_3 \exp ( - c_4 x   ) \quad 
		{\rm for\ all}\ x \geq 0 \quad {\rm and} \quad 
		n \geq 1
	\end{equation}
where $c_3,c_4>0$ do not depend on $n$. 
Thus, using the same notation $z_n^\pm$ to the above, 
namely, unique points satisfying 
$v_n''(z_n^\pm) = 0$ and $z_n^\pm \geq 0$, 
we claim that 
$z_n^\pm \to z^\pm$ where $z^\pm$ are unique points 
satisfying $z^\pm > 0$ and $\omega_\pm''(z^\pm) = 0$. 
In fact, the unique existence of $z^\pm$ is ensured by Proposition \ref{defi:2.1}. 
Furthermore, by \eqref{eq:2.18}, (V1), (f2), 
$\varphi_n \equiv 0$ in $[0,\infty)$ and 
$\omega_\pm(x) \to 0$ as $|x| \to \infty$, 
there exist $n_0$ and $R_0>0$ such that 
if $n \geq n_0$ and $x \geq R_0$, then 
	\[
		\cM^\pm (v_n'') = V(x) v_n(x) - f(v_n(x)) > 0, 
	\]
which yields $z_n^\pm \leq R_0$. Moreover, by $\omega_\pm''(0)< 0$, 
we also observe that $z_n^\pm$ never approaches to $0$. 
Thus, by the uniqueness of $z^\pm$, we have 
$z_n^\pm \to z^\pm$ and we may assume $z_n^\pm > 0$.

		Next, since $v_n$ is strictly increasing in $(-\infty,0]$ and 
strictly decreasing in $[0,\infty)$, let $y_n^\pm(s)$ and 
$z_n^\pm(s)$ be inverse functions of $v_n$ satisfying 
$y_n^\pm(s) \leq 0 \leq z_n^\pm(s)$ for 
$ 0< s \leq v_n(0)$. In particular, we have 
	\[
		y_n^\pm, z_n^\pm \in C( (0,v_n(0)] , \R ), 
		\quad v_{n}(y_n^\pm(s)) = s = v_n(z_n^\pm(s))  \quad 
		{\rm for}\ 0< s \leq v_n(0).
	\]
Moreover, $y_n^\pm, z_n^\pm$ are smooth except for at most two points 
$s=v_n(0)$ and $s=v_n(y)$ where $v_n'(y) = 0$ and $y \neq 0$. 
Set $y_n^\pm := y_n^\pm ( v_n(z_n^\pm)  )$, namely, 
$ y_n^\pm < 0$ and $v_n(y_n^\pm) = v_n(z_n^\pm)$ hold. 
Moreover, by $\omega_\pm(-x) = \omega_\pm(x)$, 
$y_n^\pm \to -z^\pm$ as $n \to \infty$. 
		Next, set 
	\[
		\begin{aligned}
			& 
			E_{n,\infty,+}(x) 
			:= \frac{\Lambda}{2} (v_n'(x))^2 + F(v_n(x)) - \frac{V_\infty}{2} v_n^2(x) 
			& &\text{for $\cM^+$} ,
			\\
			& 
			E_{n,\infty,-}(x) 
			:= \frac{\lambda}{2} (v_n'(x))^2 + F(v_n(x)) - \frac{V_\infty}{2} v_n^2(x)
			& &\text{for $\cM^-$} .
		\end{aligned}
	\]
Remark that $E_{n,\infty,\pm}(x) \to 0$ as $|x| \to \infty$. 
Then we shall prove

\medskip

\noindent
{\bf Step 2:} {\sl We have 
	\[
		\begin{aligned}
			0 & \leq - E_{n,\infty,+} (z_n^+) \leq 
			c \exp \left( - 2 x_n \sqrt{ \frac{V_\infty}{\Lambda} + \xi_0 } \right), 
			\\
			0 & \leq - E_{n,\infty,-} (z_n^-) \leq 
			c \exp \left( - 2 x_n \sqrt{ \frac{V_\infty}{\lambda} + \xi_0 } \right)
		\end{aligned}
	\]
where $c>0$ is independent of $n$ and $\xi_0>0$ the constant in \emph{(V3)}.}

%\medskip

First we notice that 
	\begin{equation}\label{eq:2.20}
		\begin{aligned}
			&E_{n,\infty,+}'(x) = v_n'(x) ( \Lambda v_n'' + f(v_n) - V_\infty v_n ) 
			\\
			&= \left\{ \begin{aligned}
				& v_n' \left[ (V_n-V_\infty) v_n - s_n \varphi_n \right]
				& & {\rm if}\ v_n''(x) \geq 0,\\
				& v_n' \left[ \frac{\Lambda}{\lambda} 
				( V_n v_n - f(v_n) - s_n \varphi_n ) + f(v_n) - V_\infty v_n  \right]
				& & {\rm if} \ v_n''(x) < 0
			\end{aligned}\right.
		\end{aligned}
	\end{equation}
and
	\begin{equation}\label{eq:2.21}
		\begin{aligned}
			&E_{n,\infty,-}'(x) = v_n'(x) ( \lambda v_n'' + f(v_n) - V_\infty v_n ) 
			\\
			&= \left\{ \begin{aligned}
			& v_n' \left[ (V_n-V_\infty) v_n - s_n \varphi_n \right]
			& & {\rm if}\ v_n''(x) \geq 0,\\
			& v_n' \left[ \frac{\lambda}{\Lambda} 
			( V_n v_n - f(v_n) - s_n \varphi_n ) + f(v_n) - V_\infty v_n  \right]
			& & {\rm if} \ v_n''(x) < 0.
			\end{aligned}\right.
		\end{aligned}
	\end{equation}
Since $v_n'' > 0$ in $(z_n^\pm,\infty)$, $v_n' \leq 0$ in $[0,\infty)$ and 
$\varphi_n \equiv 0$ in $[0,\infty)$, we get 
	\[
		(E_{n,\infty,\pm})'(x) = v_n'(V_n-V_\infty) v_n \geq 0 \quad 
		{\rm in} \ (z_n^\pm,\infty).
	\]
Hence, (V3), $z_n^\pm \to z^\pm > 0$ and \eqref{eq:2.19} give 
	\[
		\begin{aligned}
			0 & \leq - E_{n,\infty,\pm} (z_n^\pm) 
			= \int_{z_n^\pm}^\infty (E_{n,\infty,\pm})'(x) \rd x 
			= \int_{z_n^\pm}^\infty ( - v_n') v_n (V_\infty-V_n) \rd x
			\\
			& \leq \left\{ 
				\begin{aligned}
					& c \exp \left( - 2 \sqrt{\frac{V_\infty}{\Lambda} + \xi_0} (x_n + z_n^+)
					  \right)  & &\text{(for $\cM^+$)}
					  \\
					& c \exp \left( - 2 \sqrt{\frac{V_\infty}{\lambda} + \xi_0} (x_n + z_n^-)
					\right)  & &\text{(for $\cM^-$)}
				\end{aligned}
			\right.
			\\
			& \leq \left\{ 
			\begin{aligned}
			& c \exp \left( - 2 x_n \sqrt{\frac{V_\infty}{\Lambda} + \xi_0}
			\right)  & &\text{(for $\cM^+$)}
			\\
			& c \exp \left( - 2 x_n \sqrt{\frac{V_\infty}{\lambda} + \xi_0} 
			\right)  & &\text{(for $\cM^-$)}.
			\end{aligned}
			\right.
		\end{aligned}
	\]
Hence, Step 2 holds.

\medskip

\noindent
{\bf Step 3:} {\sl One has 
	\[
		E_{n,\infty,\pm} (z_n^\pm) \leq E_{n,\infty,\pm}(y_n^\pm).
	\]}

%\medskip

		Recalling $y_n^\pm \to - z^\pm$ and $x_n \to \infty$, 
we notice that for each $s \in [v_n(z_n^\pm), v_n(0)]$, 
(V2) and $y_n^\pm(s) \leq z_n^\pm (s)$ imply 
$V_n(y_n^\pm (s)) \leq V_n(z_n^\pm(s))$. 
Moreover, we may assume $\varphi_n \equiv 0$ 
in $[y_n^\pm,\infty)$. Hence, noting 
	\[
		\begin{aligned}
			& \cM^\pm(v_n'') = V_n v_n - f(v_n) \quad 
			{\rm in}\ [y_n^\pm,z_n^\pm], \quad 
			v_n''(x) < 0 \quad {\rm in} \ [0,z_n^\pm),
			\\
			& v_n(y_n^\pm(s)) = s = v_n(z_n^\pm(s)) \quad 
			{\rm for}\ s \in [v_n(z_n^\pm), v_n(0)],
		\end{aligned}
	\]
we obtain 
	\[
		v_n''(x) < 0 \quad {\rm for\ all}\ x \in (y_n^\pm,z_n^\pm).
	\]
From this it follows that $v_n'(z) < 0 < v_n'(y)$ for $y_n^\pm \leq y < 0 < z \leq z_n^\pm$ 
and $y_n^\pm, z_n^\pm \in C^1([ v_n(z_n^\pm), v_n(0) ) )$. 
Thus we see from \eqref{eq:2.20}, $v_n(z^+_n) = v_n(y_n^+)$, 
the monotonicity of $V_n$ and 
the change of variables $s = v_n(x)$ that 
	\[
		\begin{aligned}
			E_{n,\infty,+}(0) - E_{n,\infty,+}(y_{n}^+) 
			& = \int_{y_n^+}^0 (E_{n,\infty,+})'(x) \rd x 
			\\
			&= \int_{y_n^+}^0 
			\left[ \frac{\Lambda}{\lambda} \left\{ V_n v_n - f(v_n)  \right\} 
			+ f(v_n) - V_\infty v_n \right] v_n' \rd x 
			\\
			&= \int_{v_n(y_n^+)}^{v_n(0)} 
				\left[ \frac{\Lambda}{\lambda} \left\{ V_n (y_n^+(s))s - f(s)  \right\} 
				+ f(s) - V_\infty s \right] \rd s 
			\\
			&\leq \int_{v_n(z_n^+)}^{v_n(0)} 
			\left[ \frac{\Lambda}{\lambda} \left\{ V_n (z_n^+(s))s - f(s)  \right\} 
			+ f(s) - V_\infty s \right] \rd s 
			\\
			&= - \int_0^{z_n^+}
			\left[ \frac{\Lambda}{\lambda} \left\{ V_n (x)v_n - f(v_n)  \right\} 
			+ f(v_n) - V_\infty v_n \right]  v_n'\rd x
			\\
			& = - \int_0^{z_n^+} (E_{n,\infty,+})'(x) \rd x 
			= E_{n,\infty,+}(0) - E_{n,\infty,+} (z_n^+)
		\end{aligned}
	\]
Hence, $E_{n,\infty,+}(z_n^+) \leq E_{n,\infty,+}(y_n^+)$. 
In a similar way, we can prove 
$E_{n,\infty,-}(z_n^-) \leq E_{n,\infty,-}(y_n^-)$ and  
Step 3 holds.

\bigskip

In what follows, we derive the estimates for $E_{n,\infty,\pm}(y_n^\pm)$. 
First we prove

\medskip

\noindent
{\bf Step 4:} {\sl $E_{n,\infty,\pm}'(x) \leq 0$ in $(-\infty,y_n^\pm)$ 
	for sufficiently large $n$.}

\medskip

For $E_{n,\infty,+}$, by $v_n' \geq 0$ in $(-\infty,y_n^+)$ and 
\eqref{eq:2.20}, if $x < y_n^+$ and $v_n''(x) \geq 0$, then we have 
	\[
		(E_{n,\infty,+})'(x) = v_n' \left\{ (V_n-V_\infty) v_n - s_n \varphi_n \right\} 
			\leq 0 .
	\]
On the other hand, if $x < y_n^+$ and $v_n''(x) < 0$, then 
$\lambda \leq \Lambda$ gives 
	\begin{eqnarray*}
		(E_{n,\infty,+})'(x) &=& v_n' ( \Lambda v_n'' + f(v_n) - V_\infty v_n )\\ 
		&=& v_n' \left\{ (\Lambda - \lambda) v_n'' + (V_n-V_\infty) v_n 
		- s_n \varphi_n \right\} \leq 0.
	\end{eqnarray*}
Hence, $(E_{n,\infty,+})'(x) \leq 0$ in $(-\infty,y_n^+)$.

		For $\cM^-$, if $x < y_n^-$ and $v_n''(x) \geq 0$, 
then we have 
	\[
		(E_{n,\infty,-})'(x) = v_n' \left\{ (V_n-V_\infty) v_n - s_n \varphi_n \right\} 
		\leq 0.
	\]
On the other hand, we consider the case 
$v_n''(x) \leq 0$ and $x < y_n^-$. We first remark that 
for sufficiently large $n$, we have $(E_{n,\infty,-})'(x) \leq 0$ 
provided $x \in (-\infty,3\kappa_0 - x_n]$ and $v_n''(x) \leq 0$. In fact, 
it follows from \eqref{eq:2.18}, (f2), $\omega_-(x) \to 0$ as $|x| \to \infty$ and 
(V1) that one can find $n_0$ and $R_0 \geq 0$ so that 
	\[
		f(v_n) - V_\infty v_n \leq 0 \quad 
		{\rm for\ each} \ n \geq n_0 \quad {\rm and} \quad 
		x \leq - R_0. 
	\]
Since we may assume $3\kappa_0 -x_n \leq -R_0$ for $n \geq n_0$ 
due to $x_n \to \infty$, the condition $v_n''(x) \leq 0$ and $x \leq 3\kappa_0 -x_n$ 
give 
	\[
		(E_{n,\infty,-})'(x) 
		= v_n'(x) \left\{ \lambda v_n''(x) + f(v_n) - V_\infty v_n  \right\} \leq 0.
	\]
Therefore, we only consider in 
$[3\kappa_0-x_n,y_n^-]$ and remark that 
$\varphi_n \equiv 0$ on the interval.

		Next, we shall show that 
$f(v_n(x)) - V_\infty v_n(x) \leq 0$ 
when $v_n''(x) \leq 0$ and $x \in [3\kappa_0-x_n,y_n^-]$. 
Noting $v_n(y_n^-) = v_n(z_n^-)$, $v_n(x) \leq v_n(y_n^-)$ 
for $x \in [3\kappa_0-x_n,y_n^-]$ and 
	\[
		v_n''(z_n^-) = 0 = V_n(z_n^-) v_n(z_n^-) - f(v_n(z_n^-)),
	\]
we infer from (V2), (f4) and $v_n'(x) \geq 0$ in $[3\kappa_0-x_n,y_n^-]$ that 
	\[
		0 = V_n(z_n^-) - \frac{f(v_n(y_n^-))}{v_n(y_n^-)} 
		\leq V_\infty - \frac{f(v_n(x))}{v_n(x)} 
		\quad {\rm for\ all} \ x \in [3\kappa_0 - x_n,y_n^-].
	\]
Thus $f(v_n(x)) - V_\infty v_n(x) \leq 0$ in $[3\kappa_0-x_n,y_n^-]$. 
Therefore, when $x \in [3 \kappa_0-x_n,y_n^-]$ and $v_n''(x) \leq 0$, 
it follows from \eqref{eq:2.21} that 
	\[
		(E_{n,\infty,-})'(x) 
		\leq \lambda v_n'(x) v_n''(x) \leq 0.
	\]
Hence, Step 4 holds.

\medskip

\noindent
{\bf Step 5:} {\it One has }
	\[
		E_{n,\infty,\pm} (y_n^\pm) \leq 
		\int_{-3\kappa_0 - x_n}^{-2 \kappa_0 -x_n} 
		v_n' v_n (V_n - V_\infty) \rd x. 
	\]

\medskip

By Step 4, we have $(E_{n,\infty,\pm})'(x) \leq 0$ in 
$(-\infty,y_n^\pm)$. Since $E_{n,\infty,\pm}(x) \to 0$ as $x \to - \infty$, 
we obtain 
	\begin{equation}\label{eq:2.22}
		E_{n,\infty,\pm} (y^\pm_n) 
		= \int_{-\infty}^{y_n^\pm} (E_{n,\infty,\pm})'(x) \rd x 
		\leq \int_{-3\kappa_0-x_n}^{-2\kappa_0-x_n} 
		(E_{n,\infty,\pm})'(x) \rd x 
	\end{equation}
Recalling \eqref{eq:2.18}, (V1), (f2), $x_n \to \infty$ and 
$\varphi_n \equiv 0$ in $[-3\kappa_0 - x_n, -2\kappa_0 - x_n]$, 
we may assume that 
	\[
		\cM^\pm (v_n'') = V_n v_n - f(v_n) \geq 0 
		\quad {\rm in}\ [-3\kappa_0 - x_n, -2\kappa_0 - x_n].
	\]
Hence, $v_n''(x) \geq 0$ in $[-3\kappa_0 - x_n,-2\kappa_0 - x_n]$ and 
	\[
		(E_{n,\infty,\pm})'(x) = v_n'(V_n-V_\infty) v_n \quad 
		{\rm in}\ [-3\kappa_0 - x_n,-2\kappa_0 - x_n].
	\]
Thus it is easily seen from \eqref{eq:2.22} that Step 5 holds.

\medskip

\noindent
{\bf Step 6:} 
{\sl There exists a $c>0$, which is independent of $n$, such that 
	\begin{equation}\label{eq:2.23}
		\begin{aligned}
			& \min \{ v_n(x), v_n'(x)  \} 
			\geq c \exp \left( -  |x| 
			\sqrt{\frac{V_\infty}{\Lambda} + \frac{\xi_0}{2(\Lambda + 1)} } \right) 
			& &\text{\emph{for $\cM^+$}},
			\\
			& \min \{ v_n(x), v_n'(x)  \} 
			\geq c \exp \left( -  |x| 
			\sqrt{\frac{V_\infty}{\lambda} + \frac{\xi_0}{2(\lambda + 1)} } \right)
			& &\text{\emph{for $\cM^-$}}
		\end{aligned}
	\end{equation}
for all $x \leq -2 \kappa_0 - x_n$ and sufficiently large $n$. 
}

\medskip

Set 
	\[
		\psi_+ (x) := c \exp \left(  x  
		\sqrt{\frac{V_\infty}{\Lambda} + \frac{\xi_0}{2(\Lambda + 1)} } \right), 
		\quad 
		\psi_- (x) := c \exp \left(   x 
		\sqrt{\frac{V_\infty}{\lambda} + \frac{\xi_0}{2(\lambda + 1)} } \right)
	\]
where $c>0$ is chosen below. 
By \eqref{eq:2.18}, (f2) and $\lambda \leq \Lambda$, we find an $R_0>0$ such that 
	\begin{equation}\label{eq:2.24}
		\left| \frac{f(v_n(x))}{v_n(x)} \right| \leq 
		\frac{\lambda}{4(\lambda + 1)} \xi_0 
		\leq \frac{\Lambda}{4(\Lambda + 1)} \xi_0 
	\end{equation}
for all $x \leq - R_0$ and sufficiently large $n$. Fix a $c>0$ so that 
$\psi_\pm (-R_0) \leq v_n(-R_0)$ for all sufficiently large $n$. 

%By (f4) and \eqref{eq:2.24}, we also have 
%	\[
%		\left| \frac{f(\psi_\pm(x))}{\psi_\pm(x)} \right| \leq 
%		\frac{\lambda}{4(\lambda + 1)} \xi_0 
%		\leq \frac{\Lambda}{4(\Lambda + 1)} \xi_0 
%		\quad {\rm for\ each} \ x \in (-\infty,-R_0]. 
%	\]
%Thus we infer that 
We first notice that 
	\begin{eqnarray*}
			&& -\cM^+(\psi_+'') + \left\{ V_n + \frac{\Lambda}{4(\Lambda + 1)} \xi_0 
			\right\} \psi_+\\
			& = &\left\{  - V_\infty - \frac{\Lambda}{2( \Lambda + 1)} \xi_0 
			+ V_n
			+ \frac{\Lambda}{4(\Lambda + 1)} \xi_0  \right\} \psi_+ 
			\\
			&=& \left\{  V_n - V_\infty - \frac{\Lambda}{4(\Lambda + 1)} \xi_0  \right\} 
			\psi_+ \leq 0 \quad {\rm in} \  (-\infty,-R_0).
	\end{eqnarray*}
Similarly, 
	\[
		-\cM^-(\psi_-'') + \left\{ V_n + \frac{\lambda}{4(\lambda + 1)} \xi_0 \right\} 
		\psi_-\leq 0 \quad {\rm in}\ (-\infty,-R_0). 
	\]

		On the other hand, for $x \in (-\infty,-R_0)$, 
it follows from \eqref{eq:2.24} that 
	\[
		\begin{aligned}
			0 &\leq s_n \varphi_n = 
			- \cM^+(v_n'') + V_n v_n - f(v_n) 
			\leq -\cM^+(v_n'') 
			+ \left\{ V_n + \frac{\Lambda}{4(\Lambda + 1) } \xi_0 \right\} v_n,\\
			0 & \leq -\cM^-(v_n'') 
			+ \left\{ V_n + \frac{\lambda}{4(\lambda + 1) } \xi_0 \right\} v_n
		\end{aligned}
	\]
Hence, putting 
	\[
		w_{n,+}(x) := v_n - \psi_+, \quad 
		w_{n,-}(x) := v_n - \psi_-,
	\]
we have 
	\begin{eqnarray*}
		0 &\leq& - \cM^-(w_{n,+}'') + 
		\left\{ V_n + \frac{\Lambda}{4(\Lambda + 1) } \xi_0 \right\} w_{n,+}, \quad \\
		0 &\leq& -\cM^-(w_{n,-}'') 
		+ 
		\left\{ V_n + \frac{\lambda}{4(\lambda + 1) } \xi_0 \right\} w_{n,-}
	\end{eqnarray*}
in $(-\infty,-R_0)$. Recalling $w_{n,\pm}(-R_0) \geq 0$ and 
$w_{n,\pm} \to 0$ as $x \to - \infty$, 
$w_{n,\pm}$ do not have negative minima and 
we get $w_{n,\pm} \geq 0$ in $(-\infty,-R_0]$. 
Thus \eqref{eq:2.23} holds for $v_n$.

		For $v_n'$, since $\varphi_n \equiv 0$ 
in $(-\infty,-2\kappa_0-x_n)$, 
there exists a $c>0$ such that 
	\[
		\cM^\pm(v_n'') = V_n v_n - f(v_n)  \geq c v_n \quad 
		{\rm in}\ (-\infty,-2\kappa_0-x_n)
	\]
for all sufficiently large $n$. Hence, \eqref{eq:2.23} holds for $v_n''$. 
Noting 
	\[
		v_n'(x) = \int_{-\infty}^x v_n''(y) \rd y,
	\]
\eqref{eq:2.23} holds.

\medskip

\noindent
{\bf Step 7:} {\sl Conclusion (Completion of the proof for \eqref{eq:2.16})}.

\medskip

We first notice that 
by the choice of $\kappa_0>0$, one has 
	\begin{equation}\label{eq:2.25}
		\min_{[-3\kappa_0-x_n,-2\kappa_0-x_n]} (V_\infty - V_n) 
		= \min_{[-3\kappa_0,-2\kappa_0]} 
		(V_\infty - V(x)) > 0.
	\end{equation}
By Step 6, we observe that for $x \in [-3\kappa_0-x_n,-2\kappa_0-x_n]$
	\begin{equation}\label{eq:2.26}
		v_n'(x)v_n(x) \geq 
		\left\{\begin{aligned}
			& c \exp \left( -2(3\kappa_0 + x_n) 
			\sqrt{\frac{V_\infty}{\Lambda} + \frac{\xi_0}{2(\Lambda + 1)}}  \right) 
			& & \text{(for $\cM^+$)},
			\\
			& c \exp \left( -2(3\kappa_0 + x_n) 
			\sqrt{\frac{V_\infty}{\lambda} + \frac{\xi_0}{2(\lambda + 1)}}  \right) 
			& & \text{(for $\cM^-$)}. 
		\end{aligned}\right.
	\end{equation}
Therefore, using \eqref{eq:2.25}, \eqref{eq:2.26} and Step 5, we obtain 
	\[
		\begin{aligned}
			& - E_{n,\infty,+} (y_n^+) 
			\geq  c \exp \left( -2 x_n 
			\sqrt{\frac{V_\infty}{\Lambda} + \frac{\xi_0}{2(\Lambda + 1)}}  \right)
			& &\text{(for $\cM^+$)},\\
			& - E_{n,\infty,-} (y_n^-) 
			\geq c \exp \left( -2 x_n 
			\sqrt{\frac{V_\infty}{\lambda} + \frac{\xi_0}{2(\lambda + 1)}}  \right)
			& &\text{(for $\cM^-$)}
		\end{aligned}
	\]
for some $c > 0$. However, 
by Steps 2 and 3, we have a contradiction. 
Hence, we may find an $M_2>0$ so that 
$x_n \leq M_2$.

For the lower bound of $(x_n)$, by introducing 
$\tilde{u}_n(x) := u_n(-x)$, we can reduce the case into the 
case $x_n \to \infty$. Thus \eqref{eq:2.16} holds.

%\bigskip

		We finally derive a contradiction 
in order to complete the proof of Proposition \ref{defi:2.9}. 
By \eqref{eq:2.16}, we may assume $x_n \to x_0$. 
Next, from Lemmas \ref{defi:2.6} and \ref{defi:2.8}, we observe that 
if $s_n \geq \tilde{t}(f,V_\infty) $, then 
	\[
		\frac{\kappa_0^2}{4 \Lambda} s_n \leq \| u_n \|_{L^\infty(\R)} 
		\leq M_0.
	\]
Therefore, $(s_n)$ is also bounded and assume that $s_n \to s_0$. 
Thus from the equation, we also get  $u_n \to u_0$  in $C^2_{\rm loc} (\R)$,
	\begin{equation}\label{eq:2.27}
		\begin{aligned}
			& 
			- \cM^\pm (u_0'') + V u_0 = f(u_0) + s_0  \varphi_0 
			\quad {\rm in}\ \R, \quad u_0(x_0) = \max_{\R} u_0, \quad 			\\
			& 
			u_0'(y) \leq 0 \leq u_0'(x) \quad 
			{\rm for}\ x \leq x_0 \leq y, \quad 
			u_0 \geq 0 \ {\rm in} \ \R. 
		\end{aligned}
	\end{equation}
If $u_0(x_0) = 0$, namely $u_0 \equiv 0$, then 
by the monotonicity of $u_n$ ($u_n'(y) \leq 0 \leq u_n'(x)$ 
for $x \leq x_n \leq y$), we choose an $R_0>3\kappa_0$ so that 
	\[
		Vu_n - f(u_n) \geq  \frac{V_0}{2} u_n
	\]
for all $|x| \geq R_0$ and sufficiently large $n$. 
Therefore, we have 
	\[
		\cM^\pm(u_n'') = V u_n - f(u_n) \geq \frac{V_0}{2} u_n  
		\quad {\rm for\ every}\ |x| \geq R_0.
	\]
Hence, we may derive the uniform exponential decay: 
	\[
		u_n(x) \leq c \exp \left( - |x| \sqrt{\frac{V_0}{2 \Lambda}} \right)
	\]
for all $x \in \R$ and $n$. By the definition of $X_{\eta_1}$ and \eqref{eq:2.3}, 
this asserts that 
$(u_n)$ is bounded in $X_{\eta_1}$, however, 
this contradicts $\|u_n\|_{X_{\eta_1}} \to \infty$.

		Next we consider the case $u_0(x_0)>0$ and shall show that 
$\lim_{|x| \to \infty} u_0(x) = 0$. If this is true, then as in the above, 
we can derive a uniform exponential decay and get a contradiction. 
Set $u_\infty := \lim_{x \to \infty} u_0(x)$. 
Since $u_0$ is a bounded solution of \eqref{eq:2.27}, we have 
	\[
		\lim_{x \to \infty} \cM^\pm(u_0'')(x) 
		= \lim_{x\to \infty} \left( Vu_0 - f(u_0) - s_0 \varphi_0 \right)
		= V_\infty u_\infty - f(u_\infty).
	\]
Thus by \eqref{eq:2.15}, either $u_\infty = 0$ or else 
$u_\infty = s_\infty$. Let us assume $u_\infty = s_\infty$. 
From \eqref{eq:2.27}, we get $u_0 \geq s_\infty$ in $[x_0,\infty)$. 
Since $V_\infty s - f(s) < 0$ for $s > s_\infty$, 
one sees that 
	\[
		\cM^\pm (u_0'') = V u_0 - f(u_0) - s_0 \varphi_0 
		\leq V_\infty u_0 - f(u_0) - s_0 \varphi_0 \leq 0 
		\quad {\rm in}\ (x_0,\infty). 
	\]
Since $u_0'(x_0) = 0$ and $u_0(x) \to s_\infty$ as $x \to \infty$, we conclude that 
$V \equiv V_\infty$, $u_0 \equiv s_\infty$ and 
$s_0 \varphi_0 \equiv 0$ in $[x_0,\infty)$. 
Hence, by \eqref{eq:2.9} and $V \equiv V_\infty$ in $[x_0,\infty)$, we see 
$3 \kappa_0 \leq x_0$. Thus we may assume $2 \kappa_0 \leq x_n$. Now set 
	\[
		\begin{aligned}
			E_{n,+}(x) &:= \frac{\Lambda}{2} (u_n'(x))^2 
			+ F(u_n(x)) - \frac{V(x)}{2}u_n^2(x) 
			& &{\rm for}\ \cM^+, \\
			E_{n,-}(x) &:= \frac{\lambda}{2} (u_n'(x))^2 
			+ F(u_n(x)) - \frac{V(x)}{2}u_n^2(x) 
			& &{\rm for} \ \cM^-.
		\end{aligned}
	\]
Since $x_n \geq 2 \kappa_0$, $V'(x) \geq 0$ in $[x_n,\infty)$ and 
$u_n(x) \to 0$ as $|x| \to \infty$, 
arguing for the case $v_n$ in the above, 
we may find 
unique $z_n^\pm \geq x_n $ such that 
	\[
		u_n''(x) < 0 = u_n''(z_n^\pm) < u_n''(y) \quad 
		{\rm for}\ x_n \leq x < z_n^\pm < y.
	\]
Therefore, 
	\[
		(E_{n,\pm})'(x) = -\frac{V'}{2} u_n^2 \leq 0 \quad 
		{\rm in}\ [z_n^\pm,\infty), \quad 
		\lim_{x \to \infty} E_{n,\pm}(x) = 0, \quad 
		E_{n,\pm} (z_n^\pm) \geq 0.
	\]
By $u_n''(z_n^\pm) = 0$, one has 
	\[
		V_0 \leq V(z_n^\pm) = \frac{f(u_n(z_n^\pm))}{u_n(z_n^\pm)}. 
	\]
From (f2), we may find a $\delta_0>0$ so that 
$\delta_0 \leq u_n(z_n^\pm)$. 
Recalling $x_0 \leq \liminf_{n\to \infty} z_n^\pm$,  
$V \equiv V_\infty$ in $[x_0,\infty)$ and \eqref{eq:2.15}, we obtain 
	\[
		\begin{aligned}
		& u_n(z_n^\pm) \to s_\infty, \\ 
		& F(u_{n}(z_n^\pm)) - \frac{V(z_n^\pm)}{2} u_n(z_n^\pm)^2
		\to F(s_\infty) - \frac{V_\infty}{2} s_\infty^2 
		= G_\infty(s_\infty) < 0. 
		\end{aligned}
	\]
Combining with $E_{n,\pm}(z_n^\pm) \geq 0$, we may find 
a $\delta_1 > 0$ so that 
	\[
		(u_n'(z_n^\pm))^2 \geq \delta_1 .
	\]
Noting that $u_n'(x_n) = 0$ and $(u_n'')$ is bounded in $L^\infty(\R)$, 
we have $0<\delta_2 \leq z_n^\pm - x_n$ 
for some $\delta_2>0$, and 
	\[
		(u_n'(x))^2 \geq \delta_3^2 > 0 \quad 
		{\rm in} \ [z_n^\pm-\delta_4,z_n^\pm]
	\]
for some $\delta_3,\delta_4 > 0$ with $\delta_4 \leq \delta_2$. Thus 
	\[
		u_n(x_n) \geq u_n(z_n^\pm - \delta_4) 
		\geq u_n(z_n^\pm ) + \delta_3 \delta_4 
		\to s_\infty + \delta_3 \delta_4.
	\]
This contradicts $u_n(x_n) \to u_0(x_0) = s_\infty$. 
Hence, $u_\infty = 0$.

		For $\lim_{x \to - \infty} u_0(x) = 0$, 
by introducing $v_n(x) = u_n(-x)$ and $v_0(x) = u_0(-x)$, 
we can reduce into the former case and get 
$\lim_{x \to - \infty} u_0(x) = 0$. 
Now we complete the proof of Proposition \ref{defi:2.9}. 
	\end{proof}

\section{ Non-existence theorem}
\label{section:3}

In this section we prove Theorem \ref{defi:1.3} that asserts 
that the equation  \eqref{eq:1.1} 
%or \eqref{eq:1.2}
does not have a solution when $V$ is monotone.  
The argument below is similar to that of Proposition \ref{defi:2.9}. 

\begin{proof}[Proof of Theorem \ref{defi:1.3}]
	Let us suppose for  contradiction  that 
$u$ is a positive solution of \eqref{eq:1.1} %or \eqref{eq:1.2}
 and let 
$x_0$ be a maximum point of $u$.  
Noting that $V$ is non-decreasing and that the argument in Lemma \ref{defi:2.7} 
works under (f1) and \eqref{eq:1.4}, 
if $\bar x$ satisfies $u'(\bar x)=0$ then $u'(x)<0$ 
for all $x \in (\bar{x},\infty)$. 
Thus $x_0$ is the  unique critical point of $u$.

To proceed further, we make some preparations. 
Since $u$ is strictly decreasing in $(x_0,\infty)$ and $V$ non-decreasing in $\R$, 
by (f4), we see that the function
	\[
		h(x) := V(x) - \frac{f(u(x))}{u(x)} : [x_0,\infty) \to \R
	\]
is strictly increasing. Moreover, since 
$u''(x_0) \leq 0$ and $\cM^\pm(u'') = u h(x)$ in $[x_0,\infty)$, 
we have $h(x_0) \leq 0$. Hence, by $h(x) \to \overline{V}>0$ as $x \to \infty$ 
thanks to (V2'), 
there is a unique $z^\pm \geq x_0$ such that 
$h(z^\pm) = 0$. In particular, 
$u''(y) > 0 = u''(z^\pm) > u''(x) $ for all $ x_0 \leq x < z^\pm < y$. 
Recalling that $u$ is strictly increasing in $(-\infty,x_0)$ and decreasing 
in $(x_0,\infty)$, $u$ has two inverse functions $y^\pm(s)$ and $z^\pm(s)$ 
satisfying $y^\pm(s) < x_0< z^\pm(s)$ for $ 0 < s < u(x_0)$. 
Next we define  $y^\pm = y^\pm( u(z^\pm) )$ and 
	\[
		\begin{aligned}
			 &H_{+}(x) = 
			\frac{\Lambda}{2} (u'(x))^2 + F(u(x)) - \frac{V(z^+)}{2} u^2(x) 
			& &{\rm for}\ \cM^+,\\	
			 &H_{-}(x) = 
			\frac{\lambda}{2} (u'(x))^2 + F(u(x)) - \frac{V(z^-)}{2} u^2(x)
			& &{\rm for} \ \cM^-.
		\end{aligned}
	\]
To complete the proof we proceed in various steps.

\medskip

\noindent
{\bf Step 1:} {\sl $ H_\pm(z^\pm) \geq 0 $ and  if 
$V(z^\pm) < \overline{V}$, then $H_{\pm} (z^\pm) > 0$. }

\medskip

We start with 
	\begin{align}
		H_{\pm}'(x) &=  u'(x)  \left( V(x) - V(z^\pm) \right) u(x) \ 
		{\rm if} \ u''(x) \geq 0, \label{align:3.1}
		\\
		H_{+}'(x) &= u'(x) \left( \frac{\Lambda}{\lambda} 
		h(x) u(x)   + f(u(x)) - V(z^+) u(x)   \right) \ 
		{\rm if} \ u''(x) \leq 0, \label{align:3.2}
		\\
		H_{-}'(x) &=  u'(x) \left( \frac{\lambda}{\Lambda} 
						 h(x) u(x) + f(u(x)) - V(z^-) u(x)  \right) \ 
						{\rm if} \ u''(x) \leq 0.\label{align:3.3}
	\end{align}
Noticing that $u' < 0 \leq V'$ and $0\le u''$ in $(z^\pm,\infty)$, from (\ref{align:3.1})
we have
	$
		H_{\pm}'(x)\le    0$ in
		 $(z^\pm,\infty)$. In case 
 $V(z^\pm) < \overline{V}$ we additionally have $H_{\pm}'(x) \not \equiv 0$. 
From $H_\pm(x) \to 0$ as $|x| \to \infty$ we conclude that Step 1 holds.

\medskip

\noindent
{\bf Step 2:} {\sl $H_\pm(z^\pm) \leq H_\pm(y^\pm)$ and if $ V \not\equiv {\rm const.}$ 
in $[y^\pm,z^\pm]$ and $y^\pm < z^\pm$, then
$H_{\pm}(z^\pm) < H_\pm (y^\pm)$}.

\medskip

When $z^\pm = x_0$ then $z^\pm = y^\pm$ and the conclusion clearly holds. 
In case  $x_0< z^\pm$  we use arguments  similar to those of Step 3 of the proof of Proposition \ref{defi:2.9}. 
Since $u''(z^\pm (s) ) < 0$ for every $s \in (u(z^\pm), u(x_0))$, 
$u''(z^\pm) = 0 = h(z^\pm) $ and $V$ is non-decreasing, we observe that 
for each $s \in (u(z^\pm), u(x_0))$, 
	\begin{equation}\label{eq:3.4}
		h(y^\pm(s)) = V(y^\pm(s)) - \frac{f(s)}{s} 
		\leq V(z^\pm(s)) - \frac{f(s)}{s} = h(z^\pm(s)) < h(z^\pm) = 0. 
	\end{equation}
From $\cM^\pm(u'') = u(x)h(x)$ it follows that $u''(x) < 0$ in $(y^\pm, z^\pm)$. 
Hence, by \eqref{align:3.2}, \eqref{eq:3.4} and changing variables $s=u(x)$, we have 
	\[
		\begin{aligned}
			H_+(x_0) - H_+(y^+) 
			%&= \int_{y^+}^{x_0} H_+'(x) \rd x 
			%\\
			&= \int_{y^+}^{x_0} 
			\left[ \frac{\Lambda}{\lambda} 
			h(x) u(x)  + f(u(x)) - V(z^+) u(x)   \right] 
			u'(x) \rd x
			\\
			&= \int_{u(y^+)}^{u(x_0)}\left[ \frac{\Lambda}{\lambda} 
				h( y^+(s)) s + f(s) - V(z^+) s   \right] \rd s 
			\\
			& \leq \int_{u(z^+)}^{u(x_0)}\left[ \frac{\Lambda}{\lambda} 
			h( z^+(s)) s + f(s) - V(z^+) s   \right] \rd s 
			\\
			& = - \int_{x_0}^{z^+} 
			\left[ \frac{\Lambda}{\lambda} 
			h(x) u(x)  + f(u(x)) - V(z^+) u(x)   \right] 
			u'(x) \rd x
			\\
			& = - \int_{x_0}^{z^+} H_+'(x) \rd x 
			= H_+ (x_0) - H_+(z^+).
		\end{aligned}
	\]
Thus $H_+(z^+) \leq H_+(y^+)$ and  if $V \not \equiv {\rm const.}$ 
in $[y^+,z^+]$, then 
we have $V(y^+) < V(z^+)$ and $h(y^\pm(s)) < h(z^\pm(s))$ 
for $s$ close to $u(z^\pm)$. Hence, in this case, $H_+(z^+) < H_+(y^+)$ holds. 
Using \eqref{align:3.3} instead of \eqref{align:3.2}, the case of $H^-$ is treated in 
a similar way.

\medskip

\noindent
{\bf Step 3:} {\sl $H_\pm '(x) \leq 0$ in $(-\infty,y^\pm)$ and
 if $V \not \equiv {\rm const.}$ in $(-\infty,y^\pm)$, 
then $H_\pm' \not \equiv 0$}.

\medskip

		First we consider $H_+$. We observe that  $u' > 0$ and $V'(x) \geq 0$ 
in $(-\infty,y^+)$, so that when $u''(x) \geq 0$, \eqref{align:3.1} implies 
	\begin{equation}\label{eq:3.5}
		H_{+}'(x) = u'(x) ( V(x) - V(z^+) ) u(x) \leq 0. 
	\end{equation}
On the other hand, if  $u''(x) < 0$, then recalling that $\lambda \leq \Lambda $, we have
	\begin{equation}\label{eq:3.6}
		\begin{aligned}
			H_{+}'(x) &= u'( \Lambda u'' + f(u) - V(z^\pm) u  ) 
			\\
			&= u' \left\{ (\Lambda - \lambda) u'' + (V(x) - V(z^+)) u  \right\} 
			\leq 0.
		\end{aligned}
	\end{equation}
Hence, $H_{+}' \leq 0$ in $(-\infty,y^+)$. 
If $V \not\equiv {\rm const.}$ in $(-\infty,y^+)$, 
we may find $x_1 < y^+$ such that $V(x_1) < V(y^+) \leq V(z^+)$. 
Then, from  \eqref{eq:3.5} or \eqref{eq:3.6}, we have $H_+'(x_1) < 0$.

		Next we consider $H_-$. 
We have $u'>0$ for  $x < y^-$, hence if  $u''(x) \geq 0$, then we have 
	\begin{equation}\label{eq:3.7}
		H_-'(x) = u'(x) \left\{ V(x) - V(z^-) \right\} u(x) \leq 0.
	\end{equation}
On the other hand, assume that   $u''(x) < 0$. 
Since $u'>0$ in $(-\infty,y^-)$, 
we get $u(x) < u(y^-) = u(z^-)$. Therefore, by the definition of $z^-$ and 
 (f4), we find
	\[
		0 =  h(z^-)=V(z^-)-\frac{f(u(z^-))}{u(z^-)} < V(z^-)-	 \frac{f(u(x))}{u(x)},
	\]
which yields $f(u(x)) - V(z^-) u(x) < 0$. 
Thus,  from \eqref{align:3.3} and monotonicity of $V$, it follows that 
	\begin{eqnarray}\label{eq:3.8}
			H'_-(x) &=& u'(x) \left[ \frac{\lambda}{\Lambda} 
			h(x) u(x) + f(u(x)) - V(z^-) u(x) \right] \nonumber
			\\
			&= &u'(x) 
			\left[ \frac{\lambda}{\Lambda}  \left\{ V(x) - V(z^- ) \right\} u(x) \right.\nonumber\\
			&&
			\left.+ \left( 1 - \frac{\lambda}{\Lambda} \right)
			\left\{ f(u(x)) - V(z^-) u(x)   \right\} 
			\right] \leq 0.
	\end{eqnarray}
By \eqref{eq:3.7} and \eqref{eq:3.8}, we get $H'_- (x) \leq 0$ in $(-\infty,y^-)$. 
Moreover, it is easily seen that 
when $V \not \equiv {\rm const.}$ in $(-\infty,y^-)$, 
$H'_- \not \equiv 0$ holds.

\medskip

\noindent
{\bf Step 4:} {\sl Conclusion.} 

\medskip

By Steps 1--3, we get 
	\begin{equation}\label{eq:3.9}
		0 \leq H_\pm (z^\pm) \leq H_\pm (y^\pm) 
		= \int_{-\infty}^{y^\pm} H_\pm'(x) \rd x \leq 0.
	\end{equation}
However, $\underline{V} < \overline{V}$, so we have 
$V \not \equiv {\rm const.}$ in either 
$(-\infty,y^\pm)$ or $(y^\pm,z^\pm)$ or $(z^\pm,\infty)$. Consequently,   
at least one inequality in \eqref{eq:3.9} is strict, providing
 a contradiction and completing the proof. 
\end{proof}

\medskip

\noindent {\bf Acknowledgements:} 
The authors would like to thank Lawrence  Evans 
for pointing up  the variational structure of \eqref{eq:1.1}. %and \eqref{eq:1.2}. 
P.F.  was  partially supported BASAL-CMM projects. 
N.I. was  partially supported by JSPS Research Fellowships 24-2259 and 
JSPS KAKENHI Grant Number JP16K17623. 
 The second author would like to thank Universidad de Chile, where this work was started, for their hospitality.

\appendix

\section{Proof of Proposition \ref{defi:2.1}}
\label{section:A}

Here we prove Proposition \ref{defi:2.1}.

	\begin{proof}[Proof of Proposition \ref{defi:2.1}]
	We first prove the existence of solutions. For $\alpha > 0$, we consider 
	\begin{align}
		- u'' &= \Lambda^{-1} g_\infty(u) 
		& &{\rm in}\ \R, & &
		(u'(0), u(0)) = (0, \alpha),
		\label{align:A.1}\\
		- u'' &= \lambda^{-1} g_\infty(u) 
		& &{\rm in}\ \R, & &
		(u'(0), u(0)) = (0, \alpha)
		\label{align:A.2}
	\end{align}
where $g_\infty(s) := f(s) - V_\infty s$, and we 
write $u_{\Lambda,\alpha}$ and $u_{\lambda,\alpha}$ for unique solutions 
of \eqref{align:A.1} and \eqref{align:A.2}. 
By (f1)--(f4), it is well known that there exists an $\alpha_{0} > 0$ so that 
$u_{\Lambda,\alpha}(x)$ hits zero at some point $x_\alpha > 0$ 
($u_{\Lambda,\alpha}(x_\alpha) = 0$) if $\alpha > \alpha_0$, 
$u_{\Lambda,\alpha_{0}}$ is a positive solution of \eqref{align:A.1} and 
$u_{\Lambda,\alpha_{0}}(x) \to 0$ as $|x| \to \infty$, and 
$u_{\Lambda,\alpha}(x)$ a positive periodic solution of \eqref{align:A.1} when $\alpha < \alpha_0$. 
The number $\alpha_0 > 0$ is characterized by 
	\begin{equation}\label{eq:A.3}
		G_\infty(\alpha_0) = 0
	\end{equation}
and \eqref{eq:A.3} has a unique positive solution due to \eqref{eq:2.15} (or (f1)--(f4)). 
For instance, see \cite{BL-83,JT-03}.  
Therefore, \eqref{eq:A.3} yields $\alpha_0 > s_\infty$.  
The same statement holds for $u_{\lambda,\alpha}$.

	For $\mu > 0$, set 
	\[
		E[u,\mu](x) := \frac{1}{2} (u'(x))^2 + \mu^{-1} G_\infty (u(x)). 
	\]
Then it is easily seen that 
	\[
		\frac{\rd}{\rd x} E[ u_{\Lambda,\alpha} , \Lambda ] (x) \equiv 0 
		\equiv \frac{\rd}{\rd x} E[ u_{\lambda, \alpha} , \lambda ] (x) \quad 
		{\rm in} \ \R. 
	\]
In particular, since 
$u_{\Lambda,\alpha_0}(x)$,  $u_{\Lambda,\alpha_0}'(x)$, 
$u_{\lambda,\alpha_0}(x)$, $u_{\lambda,\alpha_0}'(x) \to 0$ 
as $x \to \infty$,
we have $E[u_{\Lambda,\alpha_0},\Lambda] \equiv 0 \equiv 
E[u_{\lambda,\alpha_0}, \lambda]$ in $\R$.

	Since $u_{\Lambda,\alpha_0}(0) = \alpha_0 > s_\infty$ and 
$u_{\Lambda,\alpha_0}(x) \to 0$ as $x \to \infty$, 
we may choose $x_{\Lambda} > 0$ so that 
$u_{\Lambda,\alpha_0}(x_\Lambda) = s_\infty$ and 
$u_{\Lambda,\alpha_0} (x) < s_\infty$ for every $x > x_\Lambda$. 
Recalling $G_\infty(s_\infty) < 0$ and 
$E[u_{\Lambda,\alpha_0},\Lambda](x_\Lambda) = 0$, 
we obtain $u_{\Lambda, \alpha_0}'(x_\Lambda) < 0$ and 
	\[
		\begin{aligned}
			\lambda^{-1} \min_{\R} G_\infty &< 
			\frac{1}{2} (u_{\Lambda,\alpha_0}'(x_\Lambda))^2 
			+ \lambda^{-1} G_\infty( u_{\Lambda,\alpha_0}(x_\Lambda) ) 
			\\
			&= E[u_{\Lambda,\alpha_0},\lambda] (x_\Lambda)
			< E[u_{\Lambda,\alpha_0},\Lambda](x_\Lambda) = 0.
		\end{aligned}
	\]
By \eqref{eq:A.3} and \eqref{eq:2.15}, the equation
	\begin{equation}\label{eq:A.4}
		\lambda^{-1} G_\infty(s) = E[u_{\Lambda,\alpha_0},\lambda](x_\Lambda)
		\in \left(\lambda^{-1} \min_{\R} G_\infty,0\right)
	\end{equation}
has two solutions $0<s_1 < s_\infty < s_2 < \alpha_0$.

	Now we consider $u_{\lambda,s_2}(x)$. 
Since $u_{\lambda,s_2}$ is periodic and 
$E[u_{\lambda,s_2},\lambda](x) = E[ u_{\lambda,s_2},\lambda ](0) 
= \lambda^{-1} G_\infty(s_2) < 0$ in $\R$, we observe that 
	\[
		\max_{\R} u_{\lambda,s_2} = s_2 > s_\infty > s_1 = \min_{\R} u_{\lambda,s_2}. 
	\]
Hence, we may select a $y_1>0$ so that 
$u_{\lambda,s_2}(x) > s_\infty = u_{\lambda,s_2}(y_1)$ 
for each $x \in [0,y_1)$. 
From the choice of $y_1$ and \eqref{eq:A.4}, it follows that 
	\[
		u_{\lambda,s_2}(y_1) = s_\infty = u_{\Lambda,\alpha_0}(x_\Lambda), 
		\quad 
		E[u_{\lambda,s_2},\lambda] (y_1) = \lambda^{-1} G_\infty(s_2) 
		= E[u_{\Lambda,\alpha_0}, \lambda] (x_\Lambda),
	\]
which implies $|u_{\lambda,s_2}'(y_1)| = |u_{\Lambda,\alpha_0}'(x_{\Lambda})|$. 
By $u_{\lambda,s_2}'(y_1), u_{\Lambda,\alpha_0}'(x_{\Lambda}) \leq 0$ 
due to the definition of $y_1$ and $x_{\Lambda}$, 
we obtain $u_{\lambda,s_2}'(y_1) = u_{\Lambda,\alpha_0}'(x_\Lambda)$.  
Thus, set 
	\[
		u(x) := \left\{\begin{aligned}
			& u_{\lambda,s_2}(x) & &{\rm if} \ 0 \leq x \leq y_{1},\\
			& u_{\Lambda,\alpha_0}(x - y_{1} + x_\Lambda) 
			& &{\rm if} \ y_{1} < x. 
		\end{aligned}\right.
	\]
and $u(x) := u(-x)$ for $x < 0$, it is easily seen that $u \in C^1(\R)$ and 
$u$ satisfies $-\cM^+(u'') = g_\infty(u)$ in $\R \setminus \{\pm y_{1}\}$. 
In addition, from the definition of $u$, it follows that 
	\[
		\lim_{h \downarrow 0} \frac{u'(y_1+h)-u'(y_1)}{h} = 0 = 
		\lim_{h \uparrow 0}  \frac{u'(y_1+h)-u'(y_1)}{h}.
	\]
Hence, $u \in C^2(\R)$ and 
$u$ is a solution of \eqref{eq:2.1}. 
Moreover, we observe that 
$u''(x) < 0 < u''(y)$ for all $|x| < y_1$ and 
$|y| > y_1$. 
Further, it is known that $u_{\Lambda, \alpha_0}$ decays 
exponentially, so does $u$.

	For \eqref{eq:2.2}, we start with $u_{\lambda,\alpha_0}$ 
instead of $u_{\Lambda,\alpha_0}$. Then we choose an $x_\lambda>0$ so that 
$u_{\lambda,\alpha_0}(x) < s_\infty = u_{\lambda,\alpha_0}(x_\lambda)$ 
for every $x \in (x_\lambda,\infty)$. In this case, instead of \eqref{eq:A.4}, 
we consider the equation 
	\[
		\Lambda^{-1} G_\infty(s) 
		= E[ u_{\lambda,\alpha_0} , \Lambda ] (x_\lambda) 
		> E[ u_{\lambda,\alpha_0}, \lambda ] (x_\lambda) = 0
	\]
and this equation has only one solution $s_1 > \alpha_0$  
due to \eqref{eq:2.15}. 
Let us consider $u_{\Lambda,s_1}$. 
By $s_1 > \alpha_0$, we may find a $z_1>0$ satisfying $u_{\Lambda,s_1}(z_1)=0$. 
Thus, choose a $y_1>0$ so that $u_{\Lambda,s_1}(x) > s_\infty = u_{\Lambda,s_1}(y_1)$ 
for all $x \in [0,y_1)$ and set 
	\[
		u(x) := \left\{\begin{aligned}
			& u_{\Lambda,s_1}(x) & &{\rm if} \ 0 \leq x \leq y_1,\\
			& u_{\lambda,\alpha_0} ( x - y_1 + x_\lambda ) & &{\rm if} \ 
			y_1 < x.
		\end{aligned}\right.
	\]
Then as in the above, we can check that $u \in C^2(\R)$ is a solution of \eqref{eq:2.2}, 
decays exponentially and 
$u''(x) < 0 < u''(y)$ for each $|x| < y_1$ and $|y|>y_1$.

	Next, we prove the uniqueness of solutions of \eqref{eq:2.1} and \eqref{eq:2.2}. 
Let $u_1$ be a solution of \eqref{eq:2.1} constructed in the above and 
$u$ any solution of \eqref{eq:2.1}. 
By the sign property of $g_\infty$, 
we deduce that $u(0) \geq s_\infty$. Moreover, notice that 
	\[
		-\cM^+(u'') = g_\infty(u) \quad {\rm in}\ \R \quad 
		\Leftrightarrow \quad 
		- u'' = (\cM^+)^{-1} ( g_\infty(u)  ) \quad {\rm in}\ \R
	\]
where $(\cM^+)^{-1}(s) = \Lambda^{-1} s$ if $s \geq 0$ and 
$(\cM^+)^{-1} (s) = \lambda^{-1} s$ if $s < 0$. 
Since $(\cM^+)^{-1}$ and $f$ are locally Lipschitz continuous, 
the initial value problem 
	\begin{equation}\label{eq:A.5}
		- u'' = (\cM^+)^{-1} ( g_\infty(u)  ) \quad {\rm in}\ \R, 
		\quad (u'(z),u(z)) = (\alpha_1,\alpha_2)
	\end{equation}
has a unique solution $u_{z,\alpha_1,\alpha_2}$ 
for every $z, \alpha_1 \in \R$ and $\alpha_2 > 0$. 
Since $g_\infty(s_\infty) = 0$, if $u(0) = s_\infty$, then 
we infer that $u \equiv u_{0,0,s_\infty} \equiv s_\infty$, 
which contradicts $u(x) \to 0$ as $x \to \infty$. 
Hence, $u(0) > s_\infty$.

	Now choose $z_\Lambda>0$ so that 
$u(z_\Lambda) = s_\infty > u(x)$ for all $x > z_\Lambda$. 
Then $u$ satisfies 
	\[
		-u'' = \Lambda^{-1} g_\infty(u) \quad {\rm in}\ (z_\Lambda,\infty).
	\]
Noting $E[u,\Lambda] (x) \equiv 0$ in $[z_\Lambda,\infty)$ and 
$u(z_\Lambda) = s_\infty$, we have 
	\[
		(u'(z_\Lambda), u(z_\Lambda)) = ( u_{\Lambda,\alpha_0}'(x_\Lambda), 
		u_{\Lambda,\alpha_0}(x_\Lambda) ) 
		= ( u_1'(y_1), u_1(y_1) ). 
	\]
Thus it is easily seen from the construction of $u_1$ and 
the unique solvability of the initial value problem for \eqref{eq:A.5} 
with $z = z_\Lambda$ that 
$u(x)= u_1(x+y_1 - z_\Lambda)$ in $\R$.  
Noting that 
	\[
		u_1(0) = \max_{\R} u_1 > u_1(x) \quad 
		{\rm for}\ x \neq 0, \quad 
		u(0) = \max_{\R} u,
	\]
we deduce that $y_1 = z_{\Lambda}$ and $u_1 \equiv u$. 
Hence, the uniqueness of solutions of \eqref{eq:2.1} holds. 
Similarly, we can prove the uniqueness of solutions of \eqref{eq:2.2}.

	Remark that the above argument can be applied to conclude 
$u \equiv u_1$ if $u$ satisfies $-\cM^\pm(u'') = g_\infty(u)$ in $\R$ 
with $u(0) = \max_{\R} u$, $u>0$ in $\R$ and 
$u(x) \to 0$ as either $x \to \infty$ or $x \to -\infty$. 
Thus we complete the proof. 
	\end{proof}

\end{document}